\documentclass[reqno]{amsart} 
\usepackage{amssymb}
\usepackage{amsthm}

\usepackage{graphicx}
\usepackage{subcaption}
\usepackage{amsmath}
\usepackage{paralist}
\usepackage[misc]{ifsym}
\usepackage{epsfig} 
\usepackage{epstopdf} 
\usepackage[colorlinks=true]{hyperref}
\hypersetup{urlcolor=blue, citecolor=red}
\allowdisplaybreaks
\usepackage{enumitem}
\usepackage{empheq} 
\usepackage{comment}

\mathtoolsset{showonlyrefs}

\usepackage{pgfplots}
\pgfplotsset{compat=1.18}

\usepackage{caption}

\newtheorem{theorem}{Theorem}[section]
\newtheorem{corollary}[theorem]{Corollary}

\newtheorem{lemma}[theorem]{Lemma}
\newtheorem{proposition}[theorem]{Proposition}

\theoremstyle{definition}
\newtheorem{definition}[theorem]{Definition}
\newtheorem{remark}[theorem]{Remark}
\newtheorem{assumption}[theorem]{Assumption}
\newtheorem*{notation}{Notation}
\newcommand{\ep}{\varepsilon}

\usepackage[a4paper,margin=3cm]{geometry}

\newcommand\be{\begin{equation}}
\newcommand\ee{\end{equation}}
\newcommand\bea{\begin{eqnarray}}
\newcommand\eea{\end{eqnarray}}
\newcommand\beaa{\begin{eqnarray*}}
\newcommand\eeaa{\end{eqnarray*}}
\newcommand\bay{\begin{array}}
\newcommand\eay{\end{array}}
\newcommand\ba{\begin{align}}
\newcommand\ea{\end{align}}
\newcommand\beba{\begin{equation}\left\{\begin{array}{rcl}}
\newcommand\eeba{\end{array}\right.\end{equation}}
\newcommand\bebaa{\begin{equation*}\left\{\begin{array}{rcl}}
\newcommand\eebaa{\end{array}\right.\end{equation*}}
\newcommand\beca{\begin{equation}\left\{\begin{array}{rcll}}
\newcommand\eeca{\end{array}\right.\end{equation}}
\newcommand\becaa{\begin{equation*}\left\{\begin{array}{rcll}}
\newcommand\eecaa{\end{array}\right.\end{equation*}}

\newcommand{\R}{\mathbb{R}}

\newcommand{\N}{\mathbb{N}}

\def \I{\mathrm{I}}
\def \II{\mathrm{I\hspace{-.1em}I}}
\def \III{\mathrm{I\hspace{-.1em}I\hspace{-.1em}I}}

\def \la{\left\langle}
\def \ra{\right\rangle}

\makeatletter

\newcommand{\Proofname}{Proof}
\makeatother

\title[Reaction-diffusion system approximation]{Reaction-diffusion system approximation to the fast diffusion equation} %
\author{Hideki Murakawa} 
\address{Faculty of Advanced Science and Technology, Ryukoku University, 1-5 Yokotani Seta Oe-cho Otsu Shiga 520-2194, Japan}
\email{murakawa@math.ryukoku.ac.jp}

\author{Florian Salin}
\address{IRMA, UMR 7501 Universit\'e de Strasbourg et CNRS, 7 rue René Descartes, 67084, Strasbourg, France}
\address{Inria, IRMA, Universit\'e de Strasbourg, CNRS UMR 7501, F-67000, Strasbourg, France}
\email{fsalin@unistra.fr}

\date{\today}

\begin{document}


\begin{abstract}
This paper proposes a novel reaction-diffusion system approximation tailored for singular diffusion problems, typified by the fast diffusion equation. While such approximation methods have been successfully applied to degenerate parabolic equations, their extension to singular diffusion—where the diffusion coefficient diverges at low densities—has remained unexplored. To address this, we construct an approximating semilinear system characterized by a reaction relaxation parameter and a time-derivative regularizing parameter. We rigorously establish the well-posedness of this system and derive uniform a priori estimates. Using compactness arguments, we prove the convergence of the approximate solutions to the unique weak solution of the target singular diffusion equation under three distinct asymptotic regimes: the simultaneous limit, the limit via a parabolic-elliptic system, and the limit via a uniformly parabolic equation. This approach effectively transfers the diffusion singularity into the reaction terms, yielding a highly tractable system for both theoretical analysis and computation. Finally, we present numerical experiments that validate our theoretical convergence results and demonstrate the practical efficacy of the proposed approximation scheme.
\end{abstract}

\maketitle


\section{Introduction}

The problem considered in this study is the following nonlinear diffusion equation:
\begin{equation}\label{eq:NDE}
    \frac{\partial z}{\partial t} = \Delta \beta(z),
\end{equation}
where $z = z(x,t)$ denotes an unknown physical quantity depending on the spatial variable $x \in \R^d$ $(d \in \N)$ and time $t>0$. Here, $\Delta$ is the Laplace operator with respect to the spatial variable, and $\beta$ represents a constitutive function. Although the equation \eqref{eq:NDE} appears simple in form, it describes various important phenomena in the natural sciences depending on the structure of $\beta$, and exhibits rich mathematical properties accordingly.
For example, when $\beta(z) = z^m$ $(m>1)$, equation \eqref{eq:NDE} is known as the \emph{porous medium equation}~\cite{VazquezPME2006}, which models the flow of a fluid through a porous medium. When $\beta(z) = \max\{z-1,0\} + \min\{z,0\}$, equation \eqref{eq:NDE} corresponds to the \emph{Stefan problem}~\cite{Rubinstein1971Stefan}, which describes phase transition phenomena such as the melting of ice or the solidification of water.
The right-hand side of \eqref{eq:NDE} can be rewritten as
\[
    \Delta \beta(z) = \nabla \cdot (\beta'(z)\nabla z),
\]
so that $\beta'(z)$ may be interpreted as a diffusion coefficient. In the above two examples, this coefficient vanishes at certain points or regions, meaning that the diffusion degenerates there. As a consequence, the support of the solution propagates with finite speed, leading to the appearance of a free boundary.
On the other hand, when $\beta(z)=z^m$ with $0<m<1$, one has $\beta'(z)\to \infty$ as $z \to +0$. This corresponds to a singular diffusion, where the diffusivity becomes extremely large at low densities, allowing the quantity $z$ to escape rapidly to infinity and possibly leading to \emph{finite-time extinction} of the solution~\cite{VazquezSmoothing2006}. In this case, equation \eqref{eq:NDE} is referred to as the \emph{fast diffusion equation}.
Moreover, when $\beta$ is a multivalued function, the equation \eqref{eq:NDE} describes other important models such as the \emph{nonstationary filtration problem}~\cite{AltLuckhaus1983}, which arises in flows through
porous media with pressure-dependent permeability, and the \emph{Hele--Shaw problem}~\cite{DiBenedettoFriedman1984}, which models
interface evolution in situations where viscous fluid is squeezed between narrowly spaced plates.
In addition, when $\beta$ and $z$ take vector values, the resulting \emph{cross-diffusion systems} frequently 
appear in mathematical biology, for instance in the modeling of interacting species in population 
dynamics. A notable example is the Shigesada--Kawasaki--Teramoto model~\cite{skt}, in which the diffusivities depend on the densities of different species, leading to rich and complex pattern formation phenomena.


In the analysis of the nonlinear diffusion problem \eqref{eq:NDE}, the degeneracy, singularity, and cross-diffusion nature of the diffusion coefficient pose significant obstacles. 
However, \emph{reaction-diffusion system approximation} stands as an effective method to circumvent these difficulties. 
This method reinterprets nonlinear diffusion phenomena as the limit of the interaction between reaction and diffusion, originating from the analysis of fast reaction limits by Evans~\cite{evans} and subsequent developments.
Specifically, approximation theories for the Stefan problem were established by Eymard et al.~\cite{ehhp} and Hilhorst et al.~\cite{himn}. Subsequently, Murakawa~\cite{Murakawa2007} developed a general theory of reaction-diffusion system approximation for general degenerate parabolic equations, including the porous medium equation, and extended it to cross-diffusion systems~\cite{m_lscds}. 
In these studies, the original nonlinear problems were formulated as singular limits of semilinear reaction-diffusion systems. The primary advantage of this approach lies in transferring the nonlinearity inherent in the diffusion term to the reaction term, thereby reducing the governing equations to semilinear systems and significantly improving the tractability of both analysis and numerical computation. 
Indeed, leveraging the benefits of reaction-diffusion system approximation, simple linear numerical schemes that are easy to implement and computationally inexpensive yet highly accurate have been developed and analyzed for degenerate parabolic equations and cross-diffusion systems~\cite{m_lscds,m_elscds}. In recent years, the scope of application for methods based on reaction-diffusion system approximation has expanded further, including research on approximating problems involving nonlocal terms using local problems~\cite{IshiiMurakawaTanaka2026,murakawa_tanaka2024,ninomiya2017rda}.


As discussed above, the method of reaction--diffusion system approximation has been developed primarily for degenerate parabolic equations and cross-diffusion systems. 
It has achieved notable success not only in establishing a general theoretical framework but also in several applications, including numerical analysis. 
It is therefore natural to ask whether a similar approximation framework can be developed for singular diffusion problems such as the fast diffusion equation.
However, a fundamental disparity exists in the analytical treatment between singular diffusion, where the diffusion coefficient diverges to infinity, and degenerate diffusion, where the coefficient vanishes. Consequently, within the framework of reaction-diffusion system approximation, the construction of approximate solutions for singular diffusion problems and their rigorous analysis remain unexplored. 
In this paper, we therefore undertake the construction of a novel reaction-diffusion system approximation specifically targeting singular diffusion problems. 
More concretely, we construct an appropriate approximating reaction-diffusion system, prove its convergence in the singular limit rigorously, and validate the resulting convergence behavior through numerical experiments.

The organization of this paper is as follows: In Section 2, we introduce the mathematical formulation of our problem and state the main results. Specifically, we detail the construction of the novel reaction-diffusion system approximation for the singular diffusion equation and present the three main convergence theorems. Section 3 is devoted to the mathematical analysis and the proofs of the main theorems. There, we establish the well-posedness of the approximate problems, derive uniform a priori estimates, and employ compactness arguments to rigorously pass to the respective singular limits. Finally, in Section 4, we present numerical experiments to validate our theoretical convergence results and illustrate the practical performance of the proposed approximation schemes.


\section{Formulation of the problems and main results}

In this section, we establish the mathematical framework for our study and state the main convergence results. To provide proper context, we first review the classical reaction-diffusion system approximation for degenerate parabolic equations in Section 2.1. Building upon this foundation, we construct our novel approximation scheme tailored for singular diffusion equations in Section 2.2. Finally, in Section 2.3, we rigorously define the target problem, introduce the necessary assumptions, and present the main theorems of this paper.

\subsection{Revisiting the approximation for degenerate parabolic equation}

Before constructing the reaction-diffusion system approximation for the singular diffusion problem, we recall a construction of the approximation for the degenerate parabolic equation \eqref{eq:NDE} (see, e.g., \cite{Murakawa2007}).

We assume that $\beta'$, which corresponds to the diffusion coefficient, is non-negative and bounded by $L_\beta$, and we choose a positive constant $\mu$ such that $$0 \le \beta'(s) \le L_\beta < \frac{1}{\mu} \quad \mbox{ for } s\in \R$$ holds. Under this condition, we interpret the unknown function $z$ describing the phenomenon as a superposition $z = \mu u + v$ of two states: a diffusing (mobile, active) state $u$ and a non-diffusing (immobile, inactive) state $v$.

We assume that 
$u$ diffuses with coefficient 
$1/\mu$ whereas $v$ does not diffuse, and model the exchange between these two states by a reaction term with relaxation time 
$\ep>0$. This leads to the following reaction--diffusion system:

\begin{equation} \begin{cases} \displaystyle \frac{\partial u}{\partial t} = \frac{1}{\mu} \Delta u - \frac{1}{\ep} \left( u - \beta(\mu u + v) \right), \\[8pt] \displaystyle \frac{\partial v}{\partial t} = \frac{\mu}{\ep} \left( u - \beta(\mu u + v) \right). \end{cases} \label{eq:RD_approx_DPE} \end{equation}
Multiplying the first equation by $\mu$ and adding it to the second equation cancels out the reaction terms on the right-hand side, yielding the following relation: 
\begin{equation}
    \label{eq:RD_approx_tmp1}
\frac{\partial z}{\partial t} = \Delta u. 
\end{equation}
Taking the formal limit $\ep \to 0$ in \eqref{eq:RD_approx_DPE}, we have $u = \beta(\mu u + v) = \beta(z)$ in the limit. Substituting this into \eqref{eq:RD_approx_tmp1} formally recovers the original nonlinear diffusion problem \eqref{eq:NDE}. 
In this way, the system consists of two states, one diffusing and the other non-diffusing. By rapidly changing their ratio, their superposition realizes nonlinear diffusion. 
The parameter $\mu$ can be interpreted as controlling the upper bound of the diffusion coefficient, while $\ep$ represents the relaxation time between the states.

\subsection{Construction of the approximation for singular diffusion equation}

Analogous to the above discussion, we construct the reaction-diffusion system approximation for the singular diffusion equation. We deal with the alternative form of \eqref{eq:NDE}, given by 
\begin{equation}
\label{eq:FDE}
    \frac{\partial }{\partial t}\alpha(z)=\Delta z
\end{equation}
which is obtained by rewriting the constitutive law with $\alpha = \beta^{-1}$.

We assume that $\alpha'$, which can be regarded as the coefficient of the time derivative in \eqref{eq:FDE}, is non-negative and bounded by $L_\alpha$. We then choose a positive constant $\mu$ such that 
$$0 \le \alpha'(s) \le L_\alpha < \frac{1}{\mu} \quad \mbox{ for } s\in \R.$$ 
Under this condition, we introduce variables $u$ and $v$ and define $z$ as their superposition $z = \mu u + v$.
We consider a system where $u$ evolves in time with coefficient $1/\mu$, while $v$ adjusts instantaneously (represented by the coefficient 0 for its time derivative), modeled by the following system with a relaxation parameter $\ep > 0$: \begin{equation} 
\left\{
\begin{aligned} \displaystyle \frac{1}{\mu} \frac{\partial u}{\partial t} &= \Delta u - \frac{1}{\ep} \left( u - \alpha(\mu u + v) \right), \\ \displaystyle 0 &= \Delta v + \frac{\mu}{\ep} \left( u - \alpha(\mu u + v) \right). \end{aligned} 
\right.
\label{eq:RD_approx_Alt} \end{equation}
Multiplying the first equation by $\mu$ and adding it to the second eliminates the reaction terms on the right-hand side, yielding the following relation: \begin{equation} \label{eq:RD_approx_tmp2} \frac{\partial u}{\partial t} = \Delta z. \end{equation}
Taking the formal limit $\ep \to 0$ in \eqref{eq:RD_approx_Alt}, we obtain $u  = \alpha(z)$ in the limit. Substituting this into \eqref{eq:RD_approx_tmp2} formally recovers \eqref{eq:FDE}. 

In this way, the system consists of two diffusing states, one evolving in time and the other adjusting instantaneously. By rapidly changing their ratio, their superposition realizes the singular diffusion.
The parameter $\mu$ controls the upper bound of the derivative of the constitutive function $\alpha$, while $1/\ep$ represents the exchange rate between the states.

Although we have provided a parabolic-elliptic type approximation in the above construction, we further consider a parabolic-parabolic type reaction-diffusion system approximation by introducing a relaxation in the time direction.
By replacing the zero on the left-hand side of the second equation in \eqref{eq:RD_approx_Alt} with a time derivative term multiplied by a small parameter $\xi > 0$, we obtain the following system: \begin{equation} \left\{ \begin{aligned} \frac{1}{\mu} \frac{\partial u}{\partial t} &= \Delta u - \frac{1}{\ep} \left( u - \alpha(\mu u + v) \right), \\ \xi \frac{\partial v}{\partial t} &= \Delta v + \frac{\mu}{\ep} \left( u - \alpha(\mu u + v) \right). \end{aligned} \right. \label{eq:RD_approx_PP} \end{equation}
This relaxation transforms the elliptic constraint into an evolution equation, which allows us to treat the problem as a fully coupled parabolic system.

\subsection{Assumptions and main results}
\label{subsec:ass-main}

In this subsection, we state the main results of this paper. Based on the formal derivations in the previous subsections, we now rigorously formulate the initial-boundary value problem. We focus on the nonlinear diffusion equation of the following form. Let $T > 0$ and $\Omega$ be a bounded domain in $\mathbb{R}^d$ $(d\in \N)$ with a smooth boundary $\partial\Omega$. We consider the following problem:
\begin{equation*}
\mathrm{(P)} \qquad \left\{
    \begin{aligned}
        \displaystyle \frac{\partial }{\partial t}\alpha(z) &= \Delta z & &\text{in } Q_T := \Omega \times (0, T], \\
        \displaystyle z &= 0 & &\text{on }  \partial\Omega \times (0, T], \\
        \displaystyle \alpha(z(0)) &= \alpha(z_0) & &\text{in } \Omega.
    \end{aligned}
    \right.
\end{equation*}

The solution of $\mathrm{(P)}$ is approximated by the solutions of the following reaction-diffusion system with a reaction relaxation parameter $\varepsilon > 0$, a time-derivative regularizing parameter $\xi \ge 0$, and a fixed parameter $\mu > 0$:
\begin{subequations}
\begin{empheq}[left={(\mathrm{P}_{\varepsilon,\xi}) \quad \empheqlbrace}]{align}
      \frac{1}{\mu}\frac{\partial u_{\varepsilon,\xi}}{\partial t}
        &= \Delta u_{\varepsilon,\xi}
           - \frac{1}{\varepsilon}\bigl(u_{\varepsilon,\xi}
              - \alpha(\mu u_{\varepsilon,\xi}+v_{\varepsilon,\xi})\bigr)
           && \text{in } Q_T, \label{eq:u} \\
      \xi\,\frac{\partial v_{\varepsilon,\xi}}{\partial t}
        &= \Delta v_{\varepsilon,\xi}
           + \frac{\mu}{\varepsilon}\bigl(u_{\varepsilon,\xi}
              - \alpha(\mu u_{\varepsilon,\xi}+v_{\varepsilon,\xi})\bigr)
           && \text{in } Q_T, \label{eq:v} \\
      u_{\varepsilon,\xi} &= v_{\varepsilon,\xi} = 0
           && \text{on } \partial\Omega \times (0, T], 
           \notag \\
      u_{\varepsilon,\xi}(0) &= u_0, \quad v_{\varepsilon,\xi}(0) = v_0
           && \text{in } \Omega. 
           \notag
\end{empheq}
\label{eq:RD_approx_sys}
\end{subequations}  
We define the approximate solutions for the original problem $\mathrm{(P)}$ as
$$z_{\varepsilon,\xi} := \mu u_{\varepsilon,\xi} + v_{\varepsilon,\xi}.$$

We treat solutions to Problem $(\mathrm{P}_{\varepsilon,\xi})$ as strong solutions, while solutions to $(\mathrm{P})$ are considered in the weak sense.
Throughout this paper, the bracket $\la \cdot, \cdot \ra$ denotes the duality pairing between $H^{-1}(\Omega)$ and $H^1_0(\Omega)$, as well as the standard inner product in $L^2(\Omega)$. We define the weak solution of Problem $\mathrm{(P)}$ as follows.

\begin{definition}[Weak Solution of Problem $\mathrm{(P)}$]
\label{def:weak_solution}
    A function $z$ on $Q_T$ is called a weak solution of $\mathrm{(P)}$ with an initial datum $z_0\in L^2(\Omega)$ if
    $$z \in L^2(0, T; H^1_0(\Omega)), \quad \alpha(z) \in L^\infty (0,T;L^2(\Omega)),$$
    and it satisfies
    $$ - \int_0^T \la \frac{\partial \varphi}{\partial t}, \alpha(z) \ra \, \mathrm{d}t + \int_0^T \la \nabla z, \nabla \varphi \ra \, \mathrm{d}t = \la \alpha(z_0), \varphi(0)\ra $$
    for all test functions $\varphi \in H^1(0, T; L^2(\Omega)) \cap L^2(0, T; H^1_0(\Omega))$ satisfying $\varphi(T) = 0$.
\end{definition}

To state the main convergence results, we impose the following assumptions on the data and the parameter $\mu$.

\begin{assumption} \label{ass:A}
    \leavevmode
\begin{enumerate}[label=(H\arabic*), leftmargin=*]
\item[(H1)] $\alpha: \mathbb{R} \to \mathbb{R}$ is a non-decreasing, locally Lipschitz continuous function with $\alpha(\mathbb{R})=\mathbb{R}$ and $\alpha(0)=0$.
        \item[(H2)] 
        The initial data $z_0, u_0, v_0 \in L^\infty(\Omega)$ satisfy  $u_0 = \alpha(z_0)$ and $v_0 = z_0 - \mu\alpha(z_0)$.
        \item[(H3)] Let $L_\alpha > 0$ be the Lipschitz constant of $\alpha$ on the bounded interval determined by the $L^\infty$-bounds of the initial data. We assume that $\mu$ is chosen to satisfy $L_\alpha < 1/\mu$.
    \end{enumerate}
\end{assumption}

Under Assumption $\mathrm{(H2)}$,  uniform $L^\infty$ estimates for the approximate solutions are guaranteed as will be established in Lemma \ref{lem:L2H1_Linf}. Because the solutions are confined to a bounded range, we can naturally truncate and linearly extend $\alpha$ outside this range to make it a globally Lipschitz continuous function with the Lipschitz constant $L_\alpha$. Consequently, without loss of generality, we can treat $\alpha$ as globally Lipschitz continuous in the subsequent mathematical analysis.

We investigate three different asymptotic regimes to recover the target problem $\mathrm{(P)}$, each possessing its own mathematical and practical significance. 
First, we state the convergence results regarding the simultaneous limit and the relaxation limit to the parabolic-elliptic system. 

\begin{theorem}[Simultaneous and Parabolic-Elliptic Convergence] \label{thm:convergence_target}
Suppose that Assumptions $\mathrm{(H1)}$--$\mathrm{(H3)}$ hold.
Then, both in the simultaneous limit as $\varepsilon \to 0$ and $\xi \to 0$, and in the limit as $\varepsilon \to 0$ with a fixed $\xi = 0$, the approximate solutions converge in the following sense:
\begin{align*}
    u_{\varepsilon,\xi} &\to \alpha(z), \quad \alpha(z_{\varepsilon,\xi}) \to \alpha(z) \quad \text{strongly in } L^2(Q_T), \\
    v_{\varepsilon,\xi} &\rightharpoonup z - \mu\alpha(z), \quad z_{\varepsilon,\xi} \rightharpoonup z \quad \text{weakly in } L^2(0,T;H^1_0(\Omega)).
\end{align*}
Here, $z$ is the unique weak solution of Problem $\mathrm{(P)}$.
\end{theorem}

Next, taking the relaxation limit $\xi \to 0$ for a fixed $\varepsilon > 0$ reduces the fully parabolic system to a parabolic-elliptic system. 

\begin{theorem}[Convergence of Relaxation] \label{thm:convergence_relaxation}
    Suppose that Assumptions $\mathrm{(H1)}$--$\mathrm{(H3)}$ hold. For any fixed $\varepsilon > 0$, as $\xi \to 0$, the approximate solutions converge in the following sense:
    \begin{equation*}
        u_{\varepsilon,\xi} \to u_{\varepsilon,0} \quad \text{strongly in } L^2(Q_T), \quad \text{and} \quad v_{\varepsilon,\xi} \to v_{\varepsilon,0} \quad \text{strongly in } L^2(0,T;H^1_0(\Omega)). 
    \end{equation*}
    Here, $(u_{\varepsilon,0}, v_{\varepsilon,0})$ is the solution of $(\mathrm{P}_{\varepsilon,\xi})$ with $\xi=0$.
\end{theorem}

\begin{remark}[Convergence via $(P_{\varepsilon,0})$]
    Combining Theorem \ref{thm:convergence_target} and Theorem \ref{thm:convergence_relaxation}, we can recover the solution of the target problem $\mathrm{(P)}$ by taking iterated limits. Specifically, we first take the limit $\xi \to 0$ for a fixed $\varepsilon > 0$ to obtain the parabolic-elliptic system, and subsequently take the limit $\varepsilon \to 0$ to reach the original singular diffusion problem. 
\end{remark}

In addition to the above limits, we can also consider an alternative iterated limit: first taking $\varepsilon \to 0$ for a fixed $\xi > 0$, and subsequently taking $\xi \to 0$. 

For a fixed $\xi > 0$ satisfying $\xi \le 1/\mu$, we define the regularized function $\alpha_\xi : \mathbb{R} \to \mathbb{R}$ by
\begin{equation}
    \alpha_\xi(s) := \xi s + (1 - \mu\xi)\alpha(s).
\end{equation}
We then introduce the intermediate problem for a fixed $\xi > 0$:
\begin{equation*}
\mathrm{(P_{0,\xi})} \qquad \left\{
    \begin{aligned}
        \frac{\partial }{\partial t}\alpha_\xi(z_\xi) &= \Delta z_\xi & &\text{in } Q_T, \\
        z_\xi &= 0 & &\text{on }  \partial\Omega \times (0, T], \\
                \alpha_\xi(z_\xi(0)) &= \alpha_\xi(z_0) & &\text{in } \Omega.
    \end{aligned}
    \right.
\end{equation*}
The weak solution of $\mathrm{(P_{0,\xi})}$ is defined similarly to Definition~\ref{def:weak_solution}.

\begin{theorem}[Convergence via $(P_{0,\xi})$] \label{thm:iterated_limit_eps_xi}
    Suppose that Assumptions $\mathrm{(H1)}$--$\mathrm{(H3)}$ hold. Let $\xi \in (0, 1/\mu]$ be fixed. 
    As $\varepsilon \to 0$, the approximate solutions converge in the following sense:
    \begin{align*}
        u_{\varepsilon,\xi} &\to \alpha(z_\xi), \quad v_{\varepsilon,\xi} \to z_\xi - \mu\alpha(z_\xi), \quad z_{\varepsilon,\xi} \to z_\xi\\
        &\quad \text{strongly in } L^2(Q_T) \text{ and weakly in } L^2(0,T;H^1_0(\Omega)).
    \end{align*}
    Here, $z_\xi$ is the unique weak solution of the intermediate problem $\mathrm{(P_{0,\xi})}$.
    
    Furthermore, as $\xi \to 0$, the sequence $\{z_\xi\}$ converges in the following sense:
    \begin{equation*}
        z_\xi \to z \quad \text{strongly in } L^2(Q_T) \text{ and weakly in } L^2(0,T;H^1_0(\Omega)),
    \end{equation*}
    where $z$ is the unique weak solution of the target problem $\mathrm{(P)}$.
\end{theorem}


\section{A priori Estimates and Proof of Main Theorems}
\label{sec:proof}

In this section, we provide the proofs of the main results stated in the previous section. The proofs are based on the compactness method. Our strategy is first to establish uniform a priori estimates for the solutions of the approximate problems $(\mathrm{P}_{\varepsilon,\xi})$ and $(\mathrm{P}_{\varepsilon,0})$ that are independent of the small parameters $\varepsilon$ and $\xi$. These uniform bounds allow us to extract convergent subsequences. Subsequently, we pass to the limit in the approximate equations to identify the limit functions as the weak solutions of the target problems.

\subsection{Key structural relations and auxiliary functions}
\label{subsec:key-relations}

We begin by clarifying the algebraic structure underlying the approximation and introduce some auxiliary functions. The following heuristic arguments will guide our analysis of the asymptotic behavior as $\varepsilon \to 0$.

Let us denote by $u,v,z$ the expected limits of $u_{\varepsilon,\xi}$, $v_{\varepsilon,\xi}$ and $z_{\varepsilon,\xi}$ respectively. Recalling the linear relation $z_{\varepsilon,\xi} = \mu u_{\varepsilon,\xi} + v_{\varepsilon,\xi}$, and noticing that the reaction term penalizes the deviation of $u_{\varepsilon,\xi}$ from $\alpha(z_{\varepsilon,\xi})$, is is expected in the limit $\varepsilon \to 0$ to have
\begin{equation}
    u =  \alpha(z).
    \label{eq:approx_u_alpha}
\end{equation}
Consequently, we also expect
\begin{equation}
    z \in \alpha^{-1}(u),
    \label{eq:approx_alpha_inv}
\end{equation}
where $\alpha^{-1}$ is a multivalued function if $\alpha$ is not strictly monotone.

Furthermore, using the relation $v_{\varepsilon,\xi} = z_{\varepsilon,\xi} - \mu u_{\varepsilon,\xi}$ and \eqref{eq:approx_u_alpha}, it is also expected in the limit $\varepsilon \to 0$ to have
\begin{equation}
    v = z - \mu \alpha(z),
    \label{eq:approx_v_eta}
\end{equation}
or alternatively using \eqref{eq:approx_alpha_inv}
\begin{equation}
    v \in \alpha^{-1}(u) - \mu u.
    \label{eq:approx_v_zeta}
\end{equation}

In view of these considerations, we introduce the auxiliary functions $\eta:\mathbb{R}\to\mathbb{R}$ and $\zeta:\mathbb{R}\to 2^\mathbb{R}$ defined by
\begin{gather}
    \eta(s) := s - \mu \alpha(s), \label{eq:def_eta}\\
    \zeta(s) := \alpha^{-1}(s) - \mu s \label{eq:def_zeta}.
\end{gather}
With these notations the expected relations \eqref{eq:approx_u_alpha}-\eqref{eq:approx_v_zeta} can be compactly summarized by
\begin{equation}
\begin{gathered}
u = \alpha(z),\quad z\in \alpha^{-1}(u),\\
v = \eta(z),\quad z = \eta^{-1}(v),\\
v \in \zeta(u),\quad u = \zeta^{-1}(v).
\end{gathered}
\end{equation}

Under Assumption $\mathrm{(H3)}$, $\eta$ is a non-decreasing Lipschitz continuous function with Lipschitz constant $1$, and $\zeta$ is a maximal monotone function. Therefore, under $\mathrm{(H3)}$, $\alpha, \eta, \zeta$ are all maximal monotone functions, which will be crucial in the subsequent analysis. Also, note that since the inequality is strict in $\mathrm{(H3)}$, both $\eta^{-1}$ and $\zeta^{-1}$ are Lipschitz-continuous (hence single-valued). By definition $\alpha,\eta,\zeta$ are linked by the relationships
\begin{equation}\label{eq:relations_alpha_eta_zeta}
\begin{alignedat}{2}
    \mu\alpha(z) + \eta(z) =& z\quad&\forall z\in& \mathbb{R},\\
    \alpha(\mu u + v) =& u \quad&\forall u\in&\mathbb{R},\,\forall v\in \zeta(u),\\
    \eta(z) \in& \zeta(\alpha(z)) \quad&\forall z\in&\mathbb{R}.
\end{alignedat}
\end{equation}

Finally, note that in the singular case (i.e. $\alpha'(0)=0$), $\alpha^{-1}$ and $\zeta$ are not Lipschitz-continuous, and can even be multivalued if $\alpha$ is not strictly monotone. 
To deal with this issue, we will make use of the Yosida approximation of maximal monotone functions.

\begin{notation}
    For any real maximal monotone (possibly multivalued) function $f:\mathbb{R}\rightarrow 2^\mathbb{R}$ and any $\delta>0$, we denote by $f_\delta$ the Yosida approximant of $f$. We denote by $f_0(s)$ the element of minimum absolute value in $f(s)$. To simplify the notation, we also let $f^{-1}_\delta := (f^{-1})_\delta$ and $f^{-1}_0 := (f^{-1})_0$. Finally we define $\Phi_{f}(s) := \int_0^s f_0(r)\mathrm{d}r$ for any $s\in \mathbb{R}$.
\end{notation}

The Yosida approximant $f_\delta$ of a maximal monotone function $f:\mathbb{R}\rightarrow 2^\mathbb{R}$ is defined by $f_\delta(s) := (J_\delta(s) - s)/\delta$ with $J_\delta(s)$ being the unique solution to $(J_\delta(s) - s)/\delta \in f(J_\delta(s))$. Moreover, it is single-valued, non-decreasing, $(1/\delta)$-Lipschitz continuous, and for any $s\in\mathbb{R}$ it holds $|f_\delta(s)|\leq \min |f(s)|$ and $f_\delta(s)\rightarrow f_0(s)$ as $\delta\rightarrow 0$ \cite[Proposition 2.2]{barbu}. Moreover, we also note that $\Phi_f$ is a convex function, and if $0\in f(0)$ then $\Phi_f$ is also non-negative and satisfies $\Phi_f(s)\leq s f_0(s)$.

\subsection{Unique existence of solutions to the approximate problems}
\label{subsec:existence}

Before proceeding to the a priori estimates, we establish existence, uniqueness, and regularity of solutions to the approximate problems $\mathrm{(P_{\varepsilon,\xi})}$ and $\mathrm{(P_{\varepsilon,0})}$. 
We have the following result.

\begin{proposition}
    Let $\varepsilon > 0$ and $\mu>0$ be fixed and suppose that $\mathrm{(H1)}$ and $\mathrm{(H2)}$ are satisfied. Then the following hold.
    
    {Case $\xi > 0$:} Problem $\mathrm{(P_{\varepsilon,\xi})}$ admits a unique strong solution $(u_{\varepsilon,\xi}, v_{\varepsilon,\xi})$ satisfying
    $$
    u_{\varepsilon,\xi}, v_{\varepsilon,\xi} \in C([0,\infty); L^2(\Omega)) \cap C^1((0,\infty); L^2(\Omega)) \cap C((0,\infty); H^1_0(\Omega) \cap H^2(\Omega)).
    $$
    
    {Case $\xi = 0$:} Problem $\mathrm{(P_{\varepsilon,0})}$ admits a unique strong solution $(u_{\varepsilon,0}, v_{\varepsilon,0})$ satisfying
    \begin{align*}
    u_{\varepsilon,0} &\in C([0,\infty); L^2(\Omega)) \cap C^1((0,\infty); L^2(\Omega)) \cap C((0,\infty); H^1_0(\Omega) \cap H^2(\Omega)), \\
    v_{\varepsilon,0} &\in C([0,\infty);  H^1_0(\Omega) \cap H^2(\Omega)). 
    \end{align*}
\end{proposition}

\begin{proof}
In the case $\xi>0$, the existence and uniqueness of a strong solution to $(\mathrm{P}_{\varepsilon,\xi})$, with regularity $u_{\varepsilon,\xi},v_{\varepsilon,\xi}\in C([0,\infty), L^2(\Omega))\cap C^1((0,\infty), L^2(\Omega)) \cap C((0,\infty), H^1_0(\Omega)\cap H^2(\Omega))$ can be established using standard techniques for monotone types semi-linear evolution equations, see e.g.\@~\cite{CH}, since the reaction terms are globally Lipschitz continuous.

In the case $\xi=0$, the existence of a strong solution can be proved similarly by first writing $v_{\varepsilon, 0}$ as a function of $u_{\varepsilon, 0}$. Indeed, given $u\in L^2(\Omega)$, the equation
\begin{equation}\label{eq:v_elliptic}
- \Delta v = \frac{\mu}{\varepsilon}(u - \alpha(\mu u + v)) \quad \text{in } \Omega, \quad v=0 \text{ on } \partial\Omega
\end{equation}
admits a unique weak solution $v\in H^1_0(\Omega)$ by standard variational methods, or Minty-Browder theory, thanks to the fact that $\alpha$ is non-decreasing. Let $u_1, u_2\in L^2(\Omega)$ and $v_1, v_2\in H^1_0(\Omega)$ be the corresponding weak solutions. Taking the difference of the equations satisfied by $v_1$ and $v_2$, and testing with $v_1 - v_2$ yields
$$
\begin{aligned}
\|v_1 - v_2\|_{H^1_0(\Omega)}^2 
=& \frac{\mu}{\varepsilon}\int_{\Omega} (u_1 - u_2) (v_1-v_2) - \frac{\mu}{\varepsilon}\int_\Omega (\alpha(\mu u_1 + v_1)-\alpha(\mu u_2 + v_2)) (v_1 - v_2),\\
=& \frac{\mu}{\varepsilon}\int_{\Omega} (u_1 - u_2) (v_1-v_2)\\ 
&- \frac{\mu}{\varepsilon}\int_\Omega (\alpha(\mu u_1 + v_1)-\alpha(\mu u_2 + v_2)) (\mu u_1 + v_1 - \mu u_2 - v_2) \\
&+ \frac{\mu^2}{\varepsilon}\int_\Omega (\alpha(\mu u_1 + v_1)-\alpha(\mu u_2 + v_2))(u_1 - u_2).
\end{aligned}
$$
The second term in the right-hand side is non-positive thanks to the monotonicity of $\alpha$. Furthermore, using the Lipschitz continuity of $\alpha$ in the third term and H\"older's inequality, we obtain
$$
\|v_1 - v_2\|_{H^1_0(\Omega)}^2 \leq C (\|u_1 - u_2\|_{L^2(\Omega)} \|v_1 - v_2\|_{L^2(\Omega)} + \|u_1 - u_2\|_{L^2(\Omega)}^2).
$$
Therefore, the Poincar\'e and Young inequalities yield
$$
\|v_1 - v_2\|_{L^2(\Omega)} \leq C \|u_1 - u_2\|_{L^2(\Omega)},
$$
up to relabeling the constant $C$. In particular, using again the Lipschitz continuity of $\alpha$, this shows that the $L^2(\Omega)$ norm of the right-hand side of \eqref{eq:v_elliptic} is controlled by $\|u\|_{L^2(\Omega)}$. Hence, from the regularity theory for elliptic equations, we deduce
$$
\|v_1 - v_2\|_{H^2(\Omega)} \leq C \|u_1 - u_2\|_{L^2(\Omega)}.
$$
up to relabeling the constant $C$. This shows that the operator $S: L^2(\Omega)\rightarrow H^2(\Omega)\cap H^1_0(\Omega)$ defined by $S(u) = v$, where $v$ is the unique weak solution to \eqref{eq:v_elliptic}, is Lipschitz continuous from $L^2(\Omega)$ to $H^2(\Omega)$.

To obtain existence of a strong solution to problem $(\mathrm{P}_{\varepsilon,0})$, we write equation \eqref{eq:u} with $\xi=0$ as
\begin{equation}\label{eq:u_closed_form}
\frac{1}{\mu}\frac{\partial}{\partial t} u_{\varepsilon,0} -\Delta u_{\varepsilon,0} = F(u_{\varepsilon,0}), \quad t\in (0,T), x\in \Omega,
\end{equation}
where
$$
F(u):=- \frac{1}{\varepsilon}(u - \alpha(\mu u + S(u))),\quad u\in L^2(\Omega).
$$
Due to the Lipschitz continuity of $\alpha$ and $S$, $F$ is Lipschitz continuous from $L^2(\Omega)$ to itself, and also from $H^1_0(\Omega)$ to itself. Therefore, by standard theory for monotone type semi-linear evolution equations (see e.g.\@~\cite[Prop. 4.3.4, Prop. 5.1.1]{CH}), there exists a unique strong solution $u_{\varepsilon,0}\in C([0,\infty), L^2(\Omega))\cap C^1((0,\infty), L^2(\Omega)) \cap C((0,\infty), H^1_0(\Omega)\cap H^2(\Omega))$ to \eqref{eq:u_closed_form} for any initial data $u_0\in L^2(\Omega)$. The corresponding $v_{\varepsilon,0}$ can then be obtained by $v_{\varepsilon,0} = S(u_{\varepsilon,0})$, and it satisfies $v_{\varepsilon,0}\in C([0,\infty), H^2(\Omega)\cap H^1_0(\Omega))$, thanks to the continuity of $S: L^2(\Omega)\rightarrow H^2(\Omega)\cap H^1_0(\Omega)$.
\end{proof}

\subsection{Uniqueness of the weak solution of $\mathrm{(P)}$}
Before deriving the uniform a priori estimates and proving the convergence results, we establish the uniqueness of the weak solution to the target problem $\mathrm{(P)}$, which will later ensure that the entire sequence of approximate solutions converges to a single limit.

\begin{proposition}
\label{prop:uniqueness}
    Suppose that Assumption $\mathrm{(H1)}$ holds. Then, the weak solution of Problem $\mathrm{(P)}$ with an initial datum $z_0\in L^2(\Omega)$ is unique.
\end{proposition}

\begin{proof}
Suppose that $z_1$ and $z_2$ are two weak solutions of $\mathrm{(P)}$ with the same initial datum $z_0$. 
By Definition~\ref{def:weak_solution}, we have $z_1 - z_2 \in L^2(0, T; H^1_0(\Omega))$ and $\alpha(z_1) - \alpha(z_2) \in L^\infty(0, T; L^2(\Omega))$. 
Taking the difference of the weak formulations for $z_1$ and $z_2$, we obtain
$$ \int_0^T \la \frac{\partial \varphi}{\partial t}, \alpha(z_1) - \alpha(z_2) \ra - \int_0^T \la \nabla (z_1 - z_2), \nabla \varphi \ra = 0 $$
for all test functions $\varphi \in H^1(0, T; L^2(\Omega)) \cap L^2(0, T; H^1_0(\Omega))$ satisfying $\varphi(T) = 0$.

For an arbitrary fixed $t_0 \in (0, T]$, we define a test function $\varphi$ by
$$ \varphi(t, x) := \begin{cases} \int_t^{t_0} (z_1(s, x) - z_2(s, x)) \mathrm{d}s & \text{if } 0 \le t \le t_0, \\ 0 & \text{if } t_0 < t \le T. \end{cases} $$
Since $z_1 - z_2 \in L^2(0, T; H^1_0(\Omega))$, it follows that $\varphi \in C([0, T]; H^1_0(\Omega))$ and $\frac{\partial \varphi}{\partial t} = -(z_1 - z_2)$ on $(0, t_0)$ and $0$ on $(t_0, T)$. 
Clearly, $\varphi(T) = 0$, so $\varphi$ is an admissible test function.

Substituting this $\varphi$ into the difference equation yields
$$ - \int_0^{t_0} \la z_1 - z_2, \alpha(z_1) - \alpha(z_2) \ra - \int_0^{t_0} \la \nabla (z_1 - z_2), \nabla \varphi \ra = 0. $$
Using the relation $\nabla (z_1 - z_2) = -\nabla \left(\frac{\partial \varphi}{\partial t}\right)$ on $(0, t_0)$, the second term can be rewritten as
$$ - \int_0^{t_0} \la \nabla (z_1 - z_2), \nabla \varphi \ra = \int_0^{t_0} \la \frac{\partial}{\partial t} \nabla \varphi, \nabla \varphi \ra = \frac{1}{2} \int_0^{t_0} \frac{\mathrm{d}}{\mathrm{d}t} \|\nabla \varphi(t)\|_{L^2(\Omega)}^2 \mathrm{d}t. $$
Evaluating this integral, we get
$$ \frac{1}{2} \|\nabla \varphi(t_0)\|_{L^2(\Omega)}^2 - \frac{1}{2} \|\nabla \varphi(0)\|_{L^2(\Omega)}^2 = - \frac{1}{2} \|\nabla \varphi(0)\|_{L^2(\Omega)}^2, $$
since $\varphi(t_0) = 0$. 

Therefore, the equation becomes
$$ \int_0^{t_0} \int_\Omega (z_1 - z_2)(\alpha(z_1) - \alpha(z_2)) + \frac{1}{2} \left\| \int_0^{t_0} \nabla (z_1(s) - z_2(s)) \mathrm{d}s \right\|_{L^2(\Omega)}^2 = 0. $$
Assumption $\mathrm{(H1)}$ implies that the first term is non-negative. 
Since the second term is also non-negative, both terms must identically be zero. 
In particular, the second term yields
$$ \int_0^{t_0} \nabla (z_1(s) - z_2(s)) \mathrm{d}s = 0 \quad \text{a.e. in } \Omega $$
for any $t_0 \in (0, T]$. Differentiating this with respect to $t_0$, we obtain $\nabla(z_1 - z_2) = 0$ a.e.\ in $Q_T$. 
Since $z_1 - z_2 \in H^1_0(\Omega)$, the Poincar\'e inequality gives $z_1 = z_2$ a.e.\ in $Q_T$. This completes the proof of uniqueness.
\end{proof}

\subsection{Uniform estimates}
\label{subsec:estimates}

We establish a priori estimates for the parabolic-parabolic system $(\mathrm{P}_{\varepsilon,\xi})$ with respect to $\varepsilon$ and $\xi$. In what follows, $C$ denotes a generic positive constant that depends on the domain $\Omega$, the final time $T$, and the initial data, but is independent of the parameters $\varepsilon$ and $\xi$. 
Let $(u_{\varepsilon,\xi}, v_{\varepsilon,\xi})$ be the strong solution to $(\mathrm{P}_{\varepsilon,\xi})$. We begin by deriving the following uniform $L^p$ estimates.

\begin{lemma} 
\label{lem:L2H1_Linf}
    For any $p \ge 2$, let $p'=p/(p-1)$ be the conjugate exponent of $p$. Denote $Q_t := \Omega\times (0,t)$, and  $s^{p-1} := |s|^{p-2}s$ for any $s\in\mathbb{R}$. Assume $\mathrm{(H1)}$--$\mathrm{(H3)}$ and $\xi\geq 0$. Then for any $t\in(0,\infty)$ and $p\ge 2$, the following inequalities hold
    \begin{equation}\label{eq:energy_ineq_u_Lp}
    \frac{1}{p}\|u_{\varepsilon,\xi}(t)\|_{L^p(\Omega)}^p + \frac{4\mu}{pp'}\|\nabla u_{\varepsilon,\xi}^{p/2}\|_{L^2(Q_t)}^2 
    \leq \frac{1}{p}\|u_0\|_{L^p(\Omega)}^p + \xi \int_\Omega v_0 \zeta^{-1}(v_0)^{p-1}
    \end{equation}
    and
    \begin{equation}\label{eq:energy_ineq_v_Lp}
    \frac{\xi}{p}\|v_{\varepsilon,\xi}(t)\|_{L^p(\Omega)}^p + \frac{4}{pp'}\|\nabla v_{\varepsilon,\xi}^{p/2}\|_{L^2(Q_t)}^2\\
    \leq \int_\Omega u_0\zeta_0(u_0)^{p-1} + \frac{\xi}{p}\|v_0\|_{L^p(\Omega)}^p.
    \end{equation}
    In particular, if $\xi>0$, then it holds
    \begin{gather}
    \|u_{\varepsilon,\xi}\|_{L^\infty(\Omega\times (0,\infty))} \leq \|u_0\|_{L^\infty(\Omega)} + \|\zeta^{-1}(v_0)\|_{L^\infty(\Omega)}, \label{eq:u_L_inf}\\
    \|v_{\varepsilon,\xi}\|_{L^\infty(\Omega\times (0,\infty))} \leq \|\zeta_{0}(u_0)\|_{L^\infty(\Omega)} + \|v_0\|_{L^\infty(\Omega)}\label{eq:v_L_inf}.
    \end{gather}
    If $\xi=0$, then it holds
    \begin{gather}
    \|u_{\varepsilon,0}\|_{L^\infty(\Omega\times (0,\infty))} \leq \|u_0\|_{L^\infty(\Omega)}, \label{eq:u_L_inf_xi_0}\\
    \|v_{\varepsilon,0}\|_{L^\infty(\Omega\times (0,\infty))} \leq \|\alpha_0^{-1}(u_0)\|_{L^\infty(\Omega)} + \mu\|u_0\|_{L^\infty(\Omega)}. \label{eq:v_L_inf_xi_0}
    \end{gather}
\end{lemma}
\begin{proof} 
To obtain \eqref{eq:energy_ineq_u_Lp}, we test \eqref{eq:u} by $\mu u_{\varepsilon,\xi}^{p-1}(t)$ and \eqref{eq:v} by $\phi(v_{\varepsilon,\xi}(t))$ where $\phi:=(\zeta^{-1}(\cdot))^{p-1}$, and then sum. This yields for any $s\in(0,\infty)$,
\begin{multline}\label{eq:derivation_enery_ineq}
\frac{1}{p}\frac{\mathrm{d}}{\mathrm{d}s} \|u_{\varepsilon,\xi}(s)\|_{L^p(\Omega)}^p + \frac{4\mu}{pp'} \int_\Omega |\nabla u_{\varepsilon,\xi}^{p/2}(s)|^2 + \xi \frac{\mathrm{d}}{\mathrm{d}s}\int_\Omega\Phi_{\phi}(v_{\varepsilon,\xi}(s))\\
+ \int_\Omega\phi'(v_{\varepsilon,\xi}(s))|\nabla v_{\varepsilon,\xi}(s)|^2 + \frac{\mu}{\varepsilon}\left\langle u_{\varepsilon,\xi}(s) - \alpha(z_{\varepsilon,\xi}(s)), u_{\varepsilon,\xi}^{p-1}(s) - \phi(v_{\varepsilon,\xi}(s))\right\rangle = 0.
\end{multline}
Since $\phi$ is non-decreasing and satisfies $\phi(0) = 0$, it holds $\Phi_\phi(s)\geq 0$ for any $s\in\mathbb{R}$, and 
$$
\Phi_{\phi}(s) = \int_0^s (\zeta^{-1}(r))^{p-1}\mathrm{d}r \leq s(\zeta^{-1}(s))^{p-1},\quad s\in\mathbb{R}.
$$
Therefore, integrating \eqref{eq:derivation_enery_ineq} between $0$ and $t$ yields
$$
\frac{1}{p}\|u_{\varepsilon,\xi}(t)\|_{L^p(\Omega)}^p + \frac{4\mu}{pp'}\|\nabla u_{\varepsilon,\xi}^{p/2}\|_{L^2(Q_t)}^2 + I
\leq \frac{1}{p}\|u_0\|_{L^p(\Omega)}^p + \xi \int_\Omega v_0 \zeta^{-1}(v_0)^{p-1},
$$
where
\begin{equation}\label{eq:def_I}
I = \frac{\mu}{\varepsilon}\int_0^t  \langle u_{\varepsilon,\xi}(s) - \alpha(z_{\varepsilon,\xi}(s)), u_{\varepsilon,\xi}^{p-1}(s) - \phi(v_{\varepsilon,\xi}(s))\rangle \mathrm{d}s.
\end{equation}
The integrand in the \eqref{eq:def_I} can be written for any $w\in \zeta(u_{\varepsilon,\xi}(s))$ as
$$
\langle \alpha(\mu u_{\varepsilon,\xi}(s) + w)  - \alpha(\mu u_{\varepsilon,\xi}(s) + v_{\varepsilon,\xi}(s)), \phi(w)- \phi(v_{\varepsilon,\xi}(s))\rangle,
$$
which is non-negative by the monotonicity of $\alpha$ and $\phi$.

This concludes the proof of \eqref{eq:energy_ineq_u_Lp}.
Moreover, Hölder inequality \eqref{eq:energy_ineq_u_Lp} yields
\begin{gather}
\|u_{\varepsilon,\xi}(t)\|_{L^p(\Omega)} \leq \|u_0\|_{L^p(\Omega)} + p^{1/p}\xi^{1/p}\|v_0\|_{L^p(\Omega)}^{1/p}\|\zeta^{-1}(v_0)\|_{L^p(\Omega)}^{(p-1)/p}.
\end{gather}
By letting $p\rightarrow \infty$, we obtain \eqref{eq:u_L_inf} in the case $\xi>0$, and \eqref{eq:u_L_inf_xi_0} in the case $\xi=0$. 

To prove \eqref{eq:energy_ineq_v_Lp}, we let $\delta>0$ and test \eqref{eq:u} by $\mu \psi^\delta(u_{\varepsilon,\xi}(t))$ where $\psi^\delta := (\zeta_{\delta}(\cdot))^{p-1}$, and \eqref{eq:v} by $v_{\varepsilon,\xi}^{p-1}(t)$. Summing and integrating between $0$ and $t$ we obtain as before
\begin{multline}
\int_\Omega\Phi_{\psi^\delta}(u_{\varepsilon,\xi}(t)) + \frac{\xi}{p}\|v_{\varepsilon,\xi}(t)\|_{L^p(\Omega)}^p + \frac{4}{pp'}\|\nabla v_{\varepsilon,\xi}^{p/2}\|_{L^2(Q_t)}^2 + I_\delta 
\leq \int_\Omega\Phi_{\psi^\delta}(u_0) + \frac{\xi}{p}\|v_0\|_{L^p(\Omega)}^p,
\end{multline} 
where
\begin{equation}
    I_\delta = \frac{\mu}{\varepsilon}\int_0^t \langle u_{\varepsilon,\xi}(s) - \alpha(z_{\varepsilon,\xi}(s)), \psi^\delta(u_{\varepsilon,\xi}(s)) - v_{\varepsilon,\xi}^{p-1}(s)\rangle \mathrm{d}s. \label{eq:Idelta}
\end{equation}
As above, $\Phi_{\psi^\delta}(s)\geq 0$ for any $s\in\mathbb{R}$, and
$$\int_\Omega\Phi_{\psi^\delta}(u_0)\leq \int_\Omega \Phi_{\psi^0}(u_0) \leq \int_\Omega u_0 \zeta_0(u_0)^{p-1},$$
where $\psi^0 := (\zeta_0(\cdot))^{p-1}$. 
Moreover, since $\zeta_\delta(u_{\varepsilon,\xi}) \to \zeta_0(u_{\varepsilon,\xi})$ a.e.\ as $\delta \to 0$, it follows from \eqref{eq:u_L_inf}, \eqref{eq:u_L_inf_xi_0} and the dominated convergence theorem that
\[
\psi^\delta(u_{\varepsilon,\xi}) \to \psi^0(u_{\varepsilon,\xi})
\quad \text{strongly in } L^2(Q_T) \text{ as } \delta \to 0.
\]
Therefore, we obtain
$$
\lim_{\delta \rightarrow 0} I_\delta = \frac{\mu}{\varepsilon}\int_0^t  \langle u_{\varepsilon,\xi}(s) - \alpha(z_{\varepsilon,\xi}(s)), \zeta_0(u_{\varepsilon,\xi}(s))^{p-1} - v_{\varepsilon,\xi}^{p-1}(s)\rangle \mathrm{d}s.
$$
The integrand of the right-hand side can be written 
$$
\langle \alpha(\mu u_{\varepsilon,\xi}(s) + \zeta_0(u_{\varepsilon,\xi}(s)))  - \alpha(\mu u_{\varepsilon,\xi}(s) + v_{\varepsilon,\xi}(s)), \zeta_0(u_{\varepsilon,\xi}(s))^{p-1}- v_{\varepsilon,\xi}^{p-1}(s)\rangle,
$$
which is non-negative by the monotonicity of $\alpha$ and $s\mapsto s^{p-1}$. This concludes the proof of the energy estimate \eqref{eq:energy_ineq_v_Lp}.

By Hölder inequality, \eqref{eq:energy_ineq_v_Lp} yields
\begin{gather}
\left(\frac{\xi}{p}\right)^{1/p}\|v_{\varepsilon,\xi}(t)\|_{L^p(\Omega)} \leq \|u_0\|_{L^p(\Omega)}^{1/p}\|\zeta_0(u_0)\|_{L^p(\Omega)}^{(p-1)/p} + \left(\frac{\xi}{p}\right)^{1/p}\|v_0\|_{L^p(\Omega)},
\end{gather}
and letting $p\rightarrow \infty$, we obtain \eqref{eq:v_L_inf} in the case $\xi>0$. 

Finally, to obtain \eqref{eq:v_L_inf_xi_0}, in the case $\xi=0$, we derive a maximum principle for the elliptic equation \eqref{eq:v_elliptic}. Let $u\in L^\infty(\Omega)$ and $v\in H^1_0(\Omega)$ be the solution to \eqref{eq:v_elliptic} with $u$ fixed. We define $M := \|\alpha_0^{-1}(u)\|_{L^\infty(\Omega)} + \mu \|u\|_{L^\infty(\Omega)}$, and we observe
$$
\mu u(x) + M\geq \alpha_0^{-1}(u(x)),\quad\text{for almost all }x\in\Omega,
$$
which also implies
$$
\alpha(\mu u(x) + M) \geq u(x),\quad\text{for almost all }x\in\Omega.
$$
Therefore, it holds, for any $\varphi\in H^1_0(\Omega)$ with $\varphi(x)\geq 0$ almost everywhere,
$$
\la \nabla(v - M), \nabla\varphi\ra + \frac{\mu}{\varepsilon}\la  \alpha(\mu u + v) - \alpha(\mu u + M) , \varphi \ra \leq  0.
$$
By taking $\varphi = (v - M)^+ := \max\{0, v - M\}\in H^1_0(\Omega)$ in the above inequality, and using the monotonicity of $\alpha$, we obtain
$$
v(x) \leq M,\quad\text{for almost all }x\in\Omega.
$$
In the same way, it holds
$$
\mu u(x) - M\leq \alpha_0^{-1}(u(x)),\quad\text{for almost all }x\in\Omega,
$$
which implies, for any $\varphi\in H^1_0(\Omega)$ with $\varphi(x)\geq 0$ almost everywhere,
$$
\la \nabla(v + M), \nabla\varphi \ra + \frac{\mu}{\varepsilon} \la \alpha(\mu u + v) - \alpha(\mu u - M) , \varphi \ra \geq  0.
$$
By taking $(v + M)^- := -\min\{0, v + M\}\in H^1_0(\Omega)$, and using again the monotonicity of $\alpha$, we obtain
$$
v(x) \geq -M,\quad\text{for almost all }x\in\Omega.
$$
This concludes the proof of \eqref{eq:v_L_inf_xi_0}.
\end{proof}

Next we derive a uniform estimate on $u_{\varepsilon,\xi} - \alpha(z_{\varepsilon,\xi})$ and $\eta(z_{\varepsilon,\xi}) - v_{\varepsilon,\xi}$. 

\begin{lemma}\label{lem:cvg_alpha_z}
Denote by $L_\alpha$ the Lipschitz constant of $\alpha$ and $Q_T:=\Omega\times(0,T)$. Then the following estimate holds:
\begin{equation}
\mu L_\alpha^{-1}\|\alpha(z_{\varepsilon,\xi}) - u_{\varepsilon,\xi}\|_{L^2(Q_T)}^2 + \|\eta(z_{\varepsilon,\xi}) - v_{\varepsilon,\xi}\|_{L^2(Q_T)}^2 \leq \varepsilon\left( \int_\Omega \Phi_{\alpha_0^{-1}}(u_0) + \xi\int_\Omega \Phi_{\eta^{-1}}(v_0) \right)
\end{equation}
\end{lemma}
\begin{proof}
Recalling that in the limit $\varepsilon\rightarrow 0$, it is expected that $u=\alpha(z)$ and $v =  \eta(z)$, we test \eqref{eq:u} by $\mu\alpha^{-1}_\delta(u)$ for $\delta>0$ and \eqref{eq:v} by $\eta^{-1}(v)$ and then sum to obtain for any $t\in (0,T)$
\begin{multline*}
\frac{\mathrm{d}}{\mathrm{d} t}\int_\Omega \Phi_{\alpha^{-1}_\delta}(u_{\varepsilon,\xi}(t)) + \mu\langle \nabla u_{\varepsilon,\xi}(t), \nabla \alpha^{-1}_\delta (u_{\varepsilon,\xi}(t))\rangle \\
 + \xi\frac{\mathrm{d}}{\mathrm{d} t}\int_\Omega \Phi_{\eta^{-1}}(v_{\varepsilon,\xi}(t)) + \langle \nabla v_{\varepsilon,\xi}(t), \nabla \eta^{-1} (v_{\varepsilon,\xi}(t))\rangle
+ \frac{\mu}{\varepsilon}I_\delta(t) = 0,
\end{multline*}
with 
$$
I_\delta(t) = \langle u_{\varepsilon,\xi}(t) - \alpha(z_{\varepsilon,\xi}(t)), \alpha^{-1}_\delta(u_{\varepsilon,\xi}(t)) - \eta^{-1}(v_{\varepsilon,\xi}(t)).
$$
Both the second and fourth terms on the left-hand side are non-negative due to the monotonicity of $\alpha^{-1}_\delta$ and $\eta^{-1}$, respectively, and moreover $\Phi_{\alpha_\delta^{-1}}$ and $\Phi_{\eta}$ are both non-negative. Hence, integrating between $0$ and $T$ yields
$$
\frac{\mu}{\varepsilon}J_\delta \leq \int_\Omega \Phi_{\alpha^{-1}_\delta}(u_0) + \xi \int_\Omega \Phi_{\eta^{-1}}(v_0),
$$
with
$$
J_\delta := \int_0^T \langle u_{\varepsilon,\xi}(t) - \alpha(z_{\varepsilon,\xi}(t)), \alpha^{-1}_\delta(u_{\varepsilon,\xi}(t)) - \eta^{-1}(v_{\varepsilon,\xi}(t)) \rangle \mathrm{d}t.
$$

On the one hand, due to the properties of the Yosida approximant, it holds
$$
\int_\Omega \Phi_{\alpha^{-1}_\delta}(u_0) \leq \int_\Omega \Phi_{\alpha^{-1}_0}(u_0).
$$
One the other hand, since $\alpha^{-1}_\delta(u_{\varepsilon,\xi}) \to \alpha^{-1}_0(u_{\varepsilon,\xi})$ strongly in $L^2(Q_T)$ as $\delta \to 0$, it holds
$$
\lim_{\delta \rightarrow 0} J_\delta = \int_0^T \langle u_{\varepsilon,\xi}(t) - \alpha(z_{\varepsilon,\xi}(t)), \alpha^{-1}_0(u_{\varepsilon,\xi}(t)) - \eta^{-1}(v_{\varepsilon,\xi}(t)) \rangle \mathrm{d}t =: J_0.
$$
We note that, by definition of $\eta$,
$$
u_{\varepsilon,\xi}(t) - \alpha(z_{\varepsilon,\xi}(t)) = \frac{1}{\mu}(\mu u_{\varepsilon,\xi}(t) - \mu \alpha(z_{\varepsilon,\xi}(t))) = \frac{1}{\mu} (\eta(z_{\varepsilon,\xi}(t))-v_{\varepsilon,\xi}(t)).
$$
Therefore, the integrand in $J_0$ is equal to
$$
\begin{aligned}
\langle u_{\varepsilon,\xi}(t) - \alpha(z_{\varepsilon,\xi}(t))&, \alpha^{-1}_0(u_{\varepsilon,\xi}(t)) - z_{\varepsilon,\xi}(t) + z_{\varepsilon,\xi}(t)- \eta^{-1}(v_{\varepsilon,\xi}(t)) \rangle \\
&= \langle u_{\varepsilon,\xi}(t)) - \alpha(z_{\varepsilon,\xi}(t)), \alpha^{-1}_0(u_{\varepsilon,\xi}(t)) - z_{\varepsilon,\xi}(t)\rangle\\
&+ \frac{1}{\mu} \langle \eta(z_{\varepsilon,\xi}(t)) - v_{\varepsilon,\xi}(t), z_{\varepsilon,\xi}(t) - \eta^{-1}(v_{\varepsilon,\xi}(t))\rangle.
\end{aligned}
$$
Thus, by the monotonicity and Lipschitz continuity of $\alpha$ and $\eta$ (recall that the Lipschitz constant of $\eta$ is 1), we have  
\begin{equation}\label{eq:J0_lower_bound}
\int_0^T \left(L_{\alpha}^{-1}\|u_{\varepsilon,\xi}(t)-\alpha(z_{\varepsilon,\xi}(t))\|_{L^2(\Omega)}^2 + \frac{1}{\mu}\|\eta(z_{\varepsilon,\xi}(t))-v_{\varepsilon,\xi}(t)\|_{L^2(\Omega)}^2 \right)\mathrm{d}t \leq J_0, 
\end{equation}
which concludes the proof.
\end{proof}

\subsection{Simultaneous and parabolic-elliptic convergence}
Armed with the uniform a priori estimates derived in the previous subsection, we are now ready to prove Theorem \ref{thm:convergence_target}.

\begin{proof}[Proof of Theorem~\ref{thm:convergence_target}]
Let $\{\varepsilon_n\}_{n=1}^\infty$ and $\{\xi_n\}_{n=1}^\infty$ be arbitrary sequences such that $\varepsilon_n > 0$, $\xi_n \ge 0$, $\varepsilon_n \to 0$, and $\xi_n \to 0$ as $n \to \infty$. We denote $u_n := u_{\varepsilon_n,\xi_n}$, $v_n := v_{\varepsilon_n,\xi_n}$, and $z_n := z_{\varepsilon_n,\xi_n}$. 

By the uniform estimates established in Lemma \ref{lem:L2H1_Linf}, $\{u_n\}$ and $\{v_n\}$ are uniformly bounded in $L^2(0,T;H^1_0(\Omega))$ and $L^\infty(Q_T)$. 
Multiplying \eqref{eq:u} by $\mu$ and adding it to \eqref{eq:v}, we obtain
\begin{equation}\label{eq:z_approx}
\frac{\partial}{\partial t}(u_n + \xi_n v_n) = \Delta z_n \quad \text{in } L^2(0,T; H^{-1}(\Omega)).
\end{equation}
Since $\{z_n\}$ is uniformly bounded in $L^2(0,T;H^1_0(\Omega))$, the sequence $\{\frac{\partial}{\partial t}(u_n + \xi_n v_n)\}$ is uniformly bounded in $L^2(0,T;H^{-1}(\Omega))$. 
Applying the Aubin-Lions lemma, there exist subsequences (still denoted by $n$) and functions $u, v \in L^2(0,T; H^1_0(\Omega)) \cap L^\infty(Q_T)$ such that
\begin{align}
u_n \rightharpoonup u, \quad v_n \rightharpoonup v \quad &\text{weakly in } L^2(0,T;H^1_0(\Omega)), \label{eq:cvg_u_H10} \\
u_n + \xi_n v_n \to u \quad &\text{strongly in } L^2(Q_T). \label{eq:cvg_u_strong}
\end{align}
Because $\{\xi_n v_n\}$ converges to $0$ strongly in $L^2(Q_T)$, \eqref{eq:cvg_u_strong} implies $u_n \to u$ strongly in $L^2(Q_T)$.
Setting $z := \mu u + v$, we also have $z_n \rightharpoonup z$ weakly in $L^2(0,T;H^1_0(\Omega))$.

Next, we identify $u = \alpha(z)$. Lemma \ref{lem:cvg_alpha_z} gives $\|u_n - \alpha(z_n)\|_{L^2(Q_T)} \to 0$ as $n \to \infty$. Thus, $\alpha(z_n) \to u$ strongly in $L^2(Q_T)$.
Since $\alpha$ is a maximal monotone graph and its realization in $L^2(Q_T)$ is weakly-strongly closed, the weak convergence $z_n \rightharpoonup z$ and the strong convergence $\alpha(z_n) \to u$ yield $u = \alpha(z)$ a.e.\ in $Q_T$.

Now, we show that $z$ is a weak solution of Problem $\mathrm{(P)}$ by passing to the limit in the weak formulation. Let $\varphi \in H^1(0,T;L^2(\Omega)) \cap L^2(0,T;H^1_0(\Omega))$ be a test function satisfying $\varphi(T) = 0$. 
Multiplying \eqref{eq:z_approx} by $\varphi$, integrating over $(0,T)$, and performing integration by parts with respect to time, we obtain the weak formulation for the approximate solutions:
\begin{equation} \label{eq:weak_form_n}
- \int_0^T \la \frac{\partial \varphi}{\partial t}, u_n + \xi_n v_n \ra + \int_0^T \la \nabla z_n, \nabla \varphi \ra = \la u_n(0) + \xi_n v_n(0), \varphi(0) \ra.
\end{equation}
We pass to the limit $n \to \infty$ in each term of \eqref{eq:weak_form_n}. 
For the first term, since $u_n \to \alpha(z)$ and $\xi_n v_n \to 0$ strongly in $L^2(Q_T)$, we have
$$ \lim_{n \to \infty} \int_0^T \la \frac{\partial \varphi}{\partial t}, u_n + \xi_n v_n \ra = \int_0^T \la \frac{\partial \varphi}{\partial t}, \alpha(z) \ra. $$
For the second term, the weak convergence $z_n \rightharpoonup z$ in $L^2(0,T;H^1_0(\Omega))$ yields
$$ \lim_{n \to \infty} \int_0^T \la \nabla z_n, \nabla \varphi \ra = \int_0^T \la \nabla z, \nabla \varphi \ra. $$
For the right-hand side, recalling that the initial data for the approximate problem are given by $u_n(0) = u_0$ and $v_n(0) = v_0$, we obtain
$$ \lim_{n \to \infty} \la u_0 + \xi_n v_0, \varphi(0) \ra = \la u_0, \varphi(0) \ra. $$
By the compatibility of the initial data $u_0 = \alpha(z_0)$, the limit equation becomes
$$ - \int_0^T \la \frac{\partial \varphi}{\partial t}, \alpha(z) \ra + \int_0^T \la \nabla z, \nabla \varphi \ra = \la \alpha(z_0), \varphi(0) \ra. $$
This precisely matches the definition of the weak solution of $\mathrm{(P)}$ stated in Definition~\ref{def:weak_solution}. Thus, $z$ is a weak solution.

By Proposition~\ref{prop:uniqueness}, the weak solution of $\mathrm{(P)}$ is unique. This uniqueness guarantees that the limit $z$ is independent of the choice of the subsequence. Therefore, the whole sequence $\{z_n\}$ converges to $z$. Since the sequences $\{\varepsilon_n\}$ and $\{\xi_n\}$ were arbitrary, the convergence holds as the continuous parameters $\varepsilon \to 0$ and $\xi \to 0$.
\end{proof}

\subsection{Convergence of relaxation}

The section deals with the convergence of $(\mathrm{P}_{\varepsilon, \xi})$ to $(\mathrm{P}_{\varepsilon, 0})$ as $\xi \to 0$.
In this subsection only, we fix $\varepsilon > 0$, and denote by $(u_{\xi}, v_{\xi})$ the solution to $(\mathrm{P}_{\varepsilon, \xi})$ for $\xi>0$ and by $(u, v)$ the solution to $(\mathrm{P}_{\varepsilon, 0})$.
In what follows, $C$ denotes a generic positive constant that is independent of $\xi$, but may depend on the fixed parameter $\varepsilon$.

\begin{proof}[Proof of Theorem~\ref{thm:convergence_relaxation}]
First, from the uniform estimates established in Lemma~\ref{lem:L2H1_Linf}, $(u_\xi)$ and $(v_\xi)$ are uniformly bounded in $L^2(0,T;H^1_0(\Omega))$ and $L^\infty(Q_T)$ with respect to $\xi$. 
By the equation \eqref{eq:u} for $u_\xi$, we see that $\frac{\partial u_\xi}{\partial t}$ is bounded in $L^2(0,T;H^{-1}(\Omega))$ independently of $\xi$ (though the bound depends on $\varepsilon$).
By the Aubin-Lions lemma, we can extract a subsequence such that $u_\xi \to u$ strongly in $L^2(Q_T)$ and weakly in $L^2(0,T;H^1_0(\Omega))$. 
Taking the difference between the equations for $v_\xi$ and $v$, we obtain
\begin{equation}\label{eq:v_diff}
\xi \frac{\partial v_\xi}{\partial t} = \Delta (v_\xi - v) + \frac{\mu}{\varepsilon}(u_\xi - u) - \frac{\mu}{\varepsilon}\bigl(\alpha(\mu u_\xi + v_\xi) - \alpha(\mu u + v)\bigr).
\end{equation}
Testing \eqref{eq:v_diff} by $v_\xi - v$ and integrating over $Q_T$, we have
\begin{equation}\label{eq:v_energy_diff}
\|\nabla(v_\xi - v)\|_{L^2(Q_T)}^2 = \I + \II + \III,
\end{equation}
where
\begin{align*}
\I &= - \xi \int_0^T \left\langle \frac{\partial v_\xi}{\partial t}, v_\xi - v \right\rangle, \\
\II &= \frac{\mu}{\varepsilon} \int_0^T \langle u_\xi - u, v_\xi - v \rangle, \\
\III &= - \frac{\mu}{\varepsilon} \int_0^T \langle \alpha(\mu u_\xi + v_\xi) - \alpha(\mu u + v), v_\xi - v \rangle.
\end{align*}

We estimate each term. For $\II$, by the Cauchy-Schwarz inequality and the uniform boundedness of $v_\xi$ and $v$, we have $|\II| \le C \|u_\xi - u\|_{L^2(Q_T)}$.
For $\III$, we decompose it as
\begin{align*}
\III &= - \frac{\mu}{\varepsilon} \int_0^T \langle \alpha(\mu u_\xi + v_\xi) - \alpha(\mu u + v_\xi), v_\xi - v \rangle \\
&\quad - \frac{\mu}{\varepsilon} \int_0^T \langle \alpha(\mu u + v_\xi) - \alpha(\mu u + v), v_\xi - v \rangle.
\end{align*}
Since $\alpha$ is non-decreasing, the second term is non-positive. Since $\alpha$ is Lipschitz continuous with constant $L_\alpha$, we can estimate $\III$ by bounding the first term as follows:
\begin{equation*}
\III \le \frac{\mu^2 L_\alpha}{\varepsilon}\|u_\xi - u\|_{L^2(Q_T)}\|v_\xi - v\|_{L^2(Q_T)} \le C \|u_\xi - u\|_{L^2(Q_T)}.
\end{equation*}

To estimate $\I$, we introduce a smooth approximation $v_h$ of $v$. Let $\chi_h \in C_0^\infty(0,T)$ be a cutoff function such that $0 \le \chi_h \le 1$, $\chi_h = 1$ on $[2h, T-2h]$, and $\operatorname{supp}(\chi_h) \subset (h, T-h)$. We define $v_h := \rho_h *_t (\chi_h v)$, where $*_t$ denotes the convolution with respect to time, and $\rho_h$ is a standard mollifier in time. Then, $v_h \in C_0^\infty(0,T; H^1_0(\Omega))$ satisfies $v_h(0) = v_h(T) = 0$, $v_h \to v$ in $L^2(0,T; H^1_0(\Omega))$ as $h \to 0$, and $\|\frac{\partial v_h}{\partial t}\|_{L^2(Q_T)} \le \frac{C}{h}$.
We rewrite $\I$ as $\I = \I_1 + \I_2 + \I_3$, where
\begin{align*}
\I_1 &:= - \xi \int_0^T \left\langle \frac{\partial v_\xi}{\partial t}, v_\xi \right\rangle = - \frac{\xi}{2} \|v_\xi(T)\|_{L^2(\Omega)}^2 + \frac{\xi}{2} \|v_0\|_{L^2(\Omega)}^2 \le C \xi, \\
\I_3 &:= \xi \int_0^T \left\langle \frac{\partial v_\xi}{\partial t}, v_h \right\rangle = - \xi \int_0^T \left\langle v_\xi, \frac{\partial v_h}{\partial t} \right\rangle \le \xi \|v_\xi\|_{L^2(Q_T)} \left\| \frac{\partial v_h}{\partial t} \right\|_{L^2(Q_T)} \le C \frac{\xi}{h}.
\end{align*}
For $\I_2$, we note that the right-hand side of the equation \eqref{eq:v} for $v_\xi$ is bounded in $L^2(0,T;H^{-1}(\Omega))$ independently of $\xi$. Thus, $\xi \frac{\partial v_\xi}{\partial t}$ is bounded in $L^2(0,T;H^{-1}(\Omega))$, which allows us to estimate:
\begin{align*}
\I_2 := \xi \int_0^T \left\langle \frac{\partial v_\xi}{\partial t}, v - v_h \right\rangle 
&\le \left\| \xi \frac{\partial v_\xi}{\partial t} \right\|_{L^2(0,T;H^{-1}(\Omega))} \|v - v_h\|_{L^2(0,T;H^1_0(\Omega))} \\
&\le C \|v - v_h\|_{L^2(0,T;H^1_0(\Omega))}.
\end{align*}

Combining the estimates for $\I, \II, \III$, we obtain
\begin{equation*}
\|\nabla(v_\xi - v)\|_{L^2(Q_T)}^2 \le C \left( \xi + \|v - v_h\|_{L^2(0,T;H^1_0(\Omega))} + \frac{\xi}{h} + \|u_\xi - u\|_{L^2(Q_T)} \right).
\end{equation*}
Taking the limit superior as $\xi \to 0$ for a fixed $h > 0$, and using the strong convergence $u_\xi \to u$ in $L^2(Q_T)$, we deduce
\begin{equation*}
\limsup_{\xi \to 0} \|\nabla(v_\xi - v)\|_{L^2(Q_T)}^2 \le C \|v - v_h\|_{L^2(0,T;H^1_0(\Omega))}.
\end{equation*}
Since this holds for any $h > 0$, passing to the limit as $h \to 0$ yields $\lim_{\xi \to 0} \|\nabla(v_\xi - v)\|_{L^2(Q_T)} = 0$. By the Poincar\'e inequality, $v_\xi \to v$ strongly in $L^2(0,T;H^1_0(\Omega))$.
\end{proof}

\subsection{Convergence via $(P_{0,\xi})$}
Finally, we establish the alternative iterated limit stated in Theorem \ref{thm:iterated_limit_eps_xi}, where we first send the reaction relaxation parameter $\varepsilon \to 0$ and subsequently take the time-regularization parameter $\xi \to 0$.

\begin{proof}[Proof of Theorem \ref{thm:iterated_limit_eps_xi}]
\textbf{Step 1: Convergence as $\varepsilon \to 0$ for fixed $\xi > 0$.}
Let $\xi \in (0, 1/\mu]$ be fixed, and let $\{\varepsilon_n\}_{n=1}^\infty$ be a sequence such that $\varepsilon_n \to 0$. We use the simplified notations $u_n, v_n,$ and $z_n$ as in the previous subsections. 
By Lemma~\ref{lem:L2H1_Linf}, $\{u_n\}$ and $\{v_n\}$ are uniformly bounded in $L^2(0,T;H^1_0(\Omega))$ and $L^\infty(Q_T)$, and we obtain the relation \eqref{eq:z_approx} for the fixed $\xi$. 
Applying the Aubin-Lions lemma exactly as before, we deduce that, up to a subsequence, $u_n + \xi v_n \to w_\xi$ strongly in $L^2(Q_T)$.

Using the relation $z_n = \mu u_n + v_n$ and the definition of $\alpha_\xi$, we have the algebraic identity:
\begin{equation*}
(u_n + \xi v_n) - \alpha_\xi(z_n) = (1 - \mu\xi)(u_n - \alpha(z_n)).
\end{equation*}
By Lemma \ref{lem:cvg_alpha_z}, $u_n - \alpha(z_n) \to 0$ strongly in $L^2(Q_T)$, which implies $\alpha_\xi(z_n) \to w_\xi$ strongly in $L^2(Q_T)$. 
Since $1 - \mu\xi \ge 0$, $\alpha_\xi$ is strongly monotone, and its inverse $\alpha_\xi^{-1}$ is $(1/\xi)$-Lipschitz continuous. Thus, $z_n \to \alpha_\xi^{-1}(w_\xi) =: z_\xi$ strongly in $L^2(Q_T)$. 
The strong convergence $z_n \to z_\xi$ and the continuity of $\alpha_\xi$ imply $\alpha_\xi(z_n) \to \alpha_\xi(z_\xi)$ strongly in $L^2(Q_T)$. 
Passing to the limit in the weak formulation proceeds in the exact same manner as in the proof of Theorem~\ref{thm:convergence_target}, establishing that $z_\xi$ is a weak solution of $\mathrm{(P_{0,\xi})}$. 
The uniqueness of the weak solution to this strictly parabolic problem, Proposition~\ref{prop:uniqueness}, guarantees the convergence of the whole sequence as $\varepsilon \to 0$.

\textbf{Step 2: Convergence as $\xi \to 0$.}
Let $\{\xi_m\}_{m=1}^\infty$ be a sequence tending to $0$. Since the uniform estimates in Lemma \ref{lem:L2H1_Linf} are independent of $\varepsilon$ and $\xi$, the limit sequence $\{z_{\xi_m}\}$ inherits these bounds. 
Therefore, $\{z_{\xi_m}\}$ is bounded in $L^2(0,T;H^1_0(\Omega))$ and $L^\infty(Q_T)$, and $\{\frac{\partial}{\partial t}\alpha_{\xi_m}(z_{\xi_m})\}$ is bounded in $L^2(0,T;H^{-1}(\Omega))$. 
By the Aubin-Lions lemma, there exists a subsequence and a function $w \in L^2(Q_T)$ such that $\alpha_{\xi_m}(z_{\xi_m}) \to w$ strongly in $L^2(Q_T)$.

Since $\alpha_{\xi_m}(z_{\xi_m}) = \xi_m z_{\xi_m} + (1 - \mu\xi_m)\alpha(z_{\xi_m})$ and the sequences are uniformly bounded in $L^\infty(Q_T)$, the terms multiplied by $\xi_m$ vanish in the limit, yielding $\alpha(z_{\xi_m}) \to w$ strongly in $L^2(Q_T)$. 
From this point, the identification $w = \alpha(z)$ and the passage to the limit in the weak formulation are completely identical to the corresponding steps in the proof of Theorem~\ref{thm:convergence_target}. 
Finally, the uniqueness of the weak solution $z$ (Proposition \ref{prop:uniqueness}) ensures the convergence of the whole sequence as $\xi \to 0$.
\end{proof}

\section{Numerical experiments}
\label{sec:NE}

In this section, we introduce a finite-difference approximation of $(\mathrm{P}_{\varepsilon,\xi})$ and provide numerical evidence of the convergence of $\mu u_{\varepsilon, \xi} + v_{\varepsilon, \xi}$ towards the solution of $(\mathrm{P})$ as $(\varepsilon,\xi)\to(0,0)$. Throughout, we restrict ourselves to
\begin{equation}\label{eq:alpha_FDE}
\alpha(s) = |s|^{q-2}s,\quad s\in\mathbb{R},
\end{equation}
with $q>2$, so that $(\mathrm{P})$ reduces to the fast diffusion equation. We refer to \cite{BonforteFigalli24} for a review on the Cauchy-Dirichlet problem for the fast diffusion equation.

\subsection{Numerical scheme}
We discretize $(\mathrm{P}_{\varepsilon,\xi})$ on $\Omega=(0,L)^d$ by a finite-difference scheme. In all experiments, we set $L=1$ and $q=2.5$. The spatial domain is discretized with mesh size $h$, and the time interval $[0,T]$ with time step $\Delta t$. We denote by $(x_i)_{i\in\mathcal{I}}\subset (0,L)^d$ the spatial grid points, where $\mathcal{I}:=\{0,\dots,(L/h-1)^d\}$, and by $(t^n)_{n\in\mathcal{N}}\subset[0,T]$ the temporal nodes, where $\mathcal{N}:=\{0,\dots,T/\Delta t\}$. For a discrete function $(u_i^n)_{n\in\mathcal{N},\,i\in\mathcal{I}}\subset\mathbb{R}$, $u_i^n$ approximates $u(x_i,t^n)$, and we set $\mathbf{u}^n:=(u_i^n)_{i\in\mathcal{I}}$.

Because the reaction terms in $(\mathrm{P}_{\varepsilon,\xi})$ become stiff for small $\varepsilon$, while the diffusion in the equation for $v$ is large for small $\xi$, we employ a fully implicit discretization to guarantee stability. The scheme is given by
\begin{equation}\label{eq:scheme}
\left\{
\begin{aligned}
    &\frac{1}{\mu}\frac{u^{n+1}_i - u^n_i}{\Delta t} =  [\Delta_{h} \mathbf{u}^{n+1}]_i - \frac{1}{\varepsilon}(u^{n+1}_i - \alpha(\mu u^{n+1}_i + v^{n+1}_i)),\quad n\in\mathcal{N},\,i\in\mathcal{I}\\
    &\xi \frac{v^{n+1}_i - v^{n}_i}{\Delta t} = [\Delta_{h} \mathbf{v}^{n+1}]_i + \frac{\mu}{\varepsilon}(u^{n+1}_i - \alpha(\mu u^{n+1}_i + v^{n+1}_i)),\quad n\in\mathcal{N},\,i\in\mathcal{I},\\
    &u^0_i = u_0(x_i),\quad v^0_i = v_0(x_i),\quad i\in\mathcal{I}
\end{aligned}
\right.
\end{equation}
where $\Delta_h:\mathbb{R}^{|\mathcal{I}|}\to \mathbb{R}^{|\mathcal{I}|}$ denotes the standard finite difference approximation of the Laplacian (using 3 points in 1D, 5 points in 2D) with zero Dirichlet boundary conditions. It is easily seen that the scheme \eqref{eq:scheme} is asymptotic-preserving as $\varepsilon, \xi \to 0$, i.e. the \eqref{eq:scheme} converges to an implicit scheme for the fast diffusion equation as $\varepsilon, \xi \to 0$ with fixed $\Delta t$ and $h$.

\subsection{Initial data}
For $\alpha$ defined by \eqref{eq:alpha_FDE}, problem $(\mathrm{P})$ admits separate variable solutions \cite{BerrymanHolland80}. We use these explicit profiles to quantify the convergence of $\mu u_{\varepsilon, \xi} + v_{\varepsilon,\xi}$ towards the solution of $(\mathrm{P})$ as $\varepsilon,\xi\to0$. The initial values of these separate variable solutions are obtained as solutions of the Lane-Emden-Fowler equation
\begin{equation}\label{eq:ini_data}
\left\{
\begin{aligned}
    -\Delta z &= c |z|^{q-2}z,\quad \text{in }\Omega,\\
    z &= 0\quad\text{on }\partial\Omega.
\end{aligned}
\right.
\end{equation}
for $c>0$. When $d\leq 2$, or when $d\geq 3$ and $q<2d/(d-2)$, standard variational methods yield a positive solution in $H^1_0(\Omega)\cap L^q(\Omega)$ for at least one value of $c$ \cite{Struwe}, and a scaling argument then shows existence for every $c>0$. Moreover, when $\Omega=(0,L)^d$, this positive solution is unique \cite{DamascelliGrossiPacella99}. We denote by $z_0$ the unique positive solution of \eqref{eq:ini_data} normalized by $\|z_0\|_{L^q(\Omega)}=1$ (hence $c=\|\nabla z_0\|_{L^2(\Omega)}^2$), and set
\begin{gather}
    u_0 = \alpha(z_0), \label{eq:data_ini_u0}\\
    v_0 =  z_0 - \mu\alpha(z_0), \label{eq:data_ini_v0}
\end{gather}
with $\mu$ chosen so that $\mu < 1/\|\alpha'(z_0)\|_{L^\infty(\Omega)}$. The solution of $(\mathrm{P})$ with initial datum $z_0$ is given by \cite{BerrymanHolland80}
\begin{equation}\label{eq:separate_variables}
z_\star(x,t)=\left(1-\frac{t}{T_\star}\right)_+^{1/(q-2)}z_0(x),\quad t\in\mathbb{R}_+,\ x\in\Omega,
\end{equation}
where $(s)_+:=\max(s,0)$ denotes the positive part of $s$ and 
$$
T_\star:= \frac{q-1}{q-2}\|\nabla z_0\|_{L^2(\Omega)}^{-2}.
$$
In particular, solutions of the form \eqref{eq:separate_variables} \emph{extinguish in finite time}. This property in fact holds for more general initial data and is a characteristic feature of the fast diffusion equation \cite{BonforteFigalli24}. To compute $u_0$ and $v_0$, we solve \eqref{eq:ini_data} by Newton's method on a refined mesh, using a spatial step ten times smaller than the value of $h$ used in \eqref{eq:scheme}, and with a tolerance of $10^{-10}$ on the increment. The initial datum $z_0$ on $\Omega = (0,L)^d$ for $d=1,2$ is represented in Figure \ref{fig:ini_data}. The associated extinction time is $T_\star\approx 0.326$ for $d=1$ and $T_\star\approx 0.175$ for $d=2$. Moreover, we take $\mu=0.5$ for $d=1$, and $\mu=0.4$ for $d=2$.

\begin{figure}[htbp]
\centering
\hspace{0.7cm}
\begin{subfigure}{0.43\textwidth}
\centering
\includegraphics[width=\linewidth]{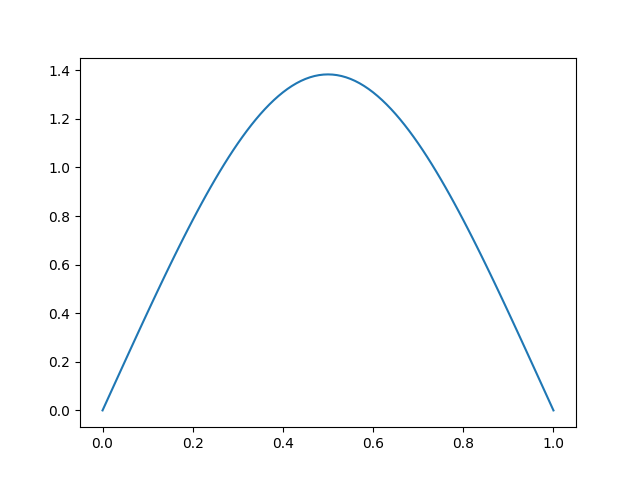}
\caption{$d=1$}
\end{subfigure}
\hfill
\begin{subfigure}{0.45\textwidth}
\centering
\includegraphics[width=\linewidth]{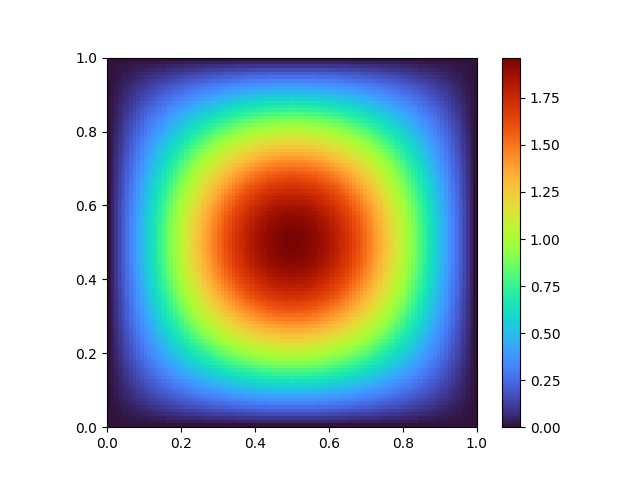}
\caption{$d=2$}
\end{subfigure}

\caption{Initial datum $z_0$ in one and two spatial dimensions. The initial data $u_0$ and $v_0$ for \eqref{eq:scheme} are taken accordingly to \eqref{eq:data_ini_u0},\eqref{eq:data_ini_v0}. $z_0$ is the unique positive solution to \eqref{eq:ini_data} with normalization $\|z_0\|_{L^q(\Omega)}=1$.}
\label{fig:ini_data}
\end{figure}

\subsection{Convergence to $(\mathrm{P})$}
In this subsection, we numerically test the convergence as $\varepsilon\to 0$ of $\mu u_{\varepsilon,\xi} + v_{\varepsilon, \xi}$ to the solution of $(\mathrm{P})$. The solution to the scheme \eqref{eq:scheme} is computed, with initial data given by \eqref{eq:data_ini_u0}, \eqref{eq:data_ini_v0}, and compared, on the nodes of the discretization, to the solution $z_*$ \eqref{eq:separate_variables} of $(\mathrm{P})$ with initial data $z_0$. In all tests, we take $\Delta t = 0.0001$ and $h=0.01$. For a fixed $\Delta t$ and $h$ the discrete solution to \eqref{eq:scheme} does not converge to the continuous solution $z_\star$ as $\varepsilon,\xi\to 0$. Instead it converges to the discrete solution of an implicit scheme for the fast diffusion equation. Nevertheless, $\Delta$ and $h$ are taken sufficiently small to observe convergence to the continuous solution $z_\star$ for values of $\varepsilon$ down to $10^{-4}$.

Results are represented in Figure \ref{fig:cvg_1d} in the one dimensional case and Figure \ref{fig:cvg_2d} in the two dimensional case. It shows a convergence of order $1$ in $\varepsilon$, both for the cases $\xi=0$ and $\xi=\varepsilon$, for the $\ell^2_{\Delta t, h}(Q_T)$ norm, where
$$
\|(z_i^n)_{n\in\mathcal{N}, i\in\mathcal{I}}\|_{\ell^2_{\Delta t, h}(Q_T)}^2:=\sum_{i\in\mathcal{I}, n\in\mathcal{N}}|z_i^n|^2 h^d\Delta t.
$$

\begin{figure}
    \centering
    \begin{subfigure}{0.48\textwidth}
        \centering

        \begin{tikzpicture}
        \begin{axis}[
            width=\linewidth,
            xmode=log,
            ymode=log,
            xlabel=$\varepsilon$,
            grid=major,
            legend pos=south east
        ]

        \addplot table [col sep=comma,x=eps,y=L2_error] {L2_error_xi_0.csv};
        \addlegendentry{$\|z_{\varepsilon,0} - z_\star\|_{\ell^2_{\Delta t, h}(Q_T)}$ }
        
        \addplot[dashed,domain=5e-4:1e-1] {x};
        \addlegendentry{$\varepsilon$}
        
        \end{axis}
        \end{tikzpicture}
        
        \caption{$\xi=0$}
        \label{subfig:1d_xi_0}
    \end{subfigure}%
    \hfill
    \begin{subfigure}{0.48\textwidth}
        \centering
        \begin{tikzpicture}
        \begin{axis}[
            width=\linewidth,
            xmode=log,
            ymode=log,
            xlabel=$\varepsilon$,
            grid=major,
            legend pos=south east
        ]
        
        \addplot table [col sep=comma,x=eps,y=L2_error] {L2_error_xi_eps.csv};
        \addlegendentry{$\|z_{\varepsilon,\xi} - z_\star\|_{\ell^2_{\Delta t, h}(Q_T)}$ }
        
        \addplot[dashed,domain=5e-4:1e-1] {x};
        \addlegendentry{$\varepsilon$}
        
        \end{axis}
        \end{tikzpicture}
        
        \caption{$\xi=\varepsilon$}
        \label{subfig:1d_xi_eps}
    \end{subfigure}
    
    \caption{Convergence for $\varepsilon\to 0$ and $\xi=0$ (Figure \ref{subfig:1d_xi_0}) or $\xi=\varepsilon$ (Figure \ref{subfig:1d_xi_eps}). $\mu=0.5$, $q=2.5$, $T=0.6$, $\Omega = (0,1)$, $\Delta t = 10^{-4}$, $\Delta x = 10^{-2}$.}
    \label{fig:cvg_1d}
\end{figure}
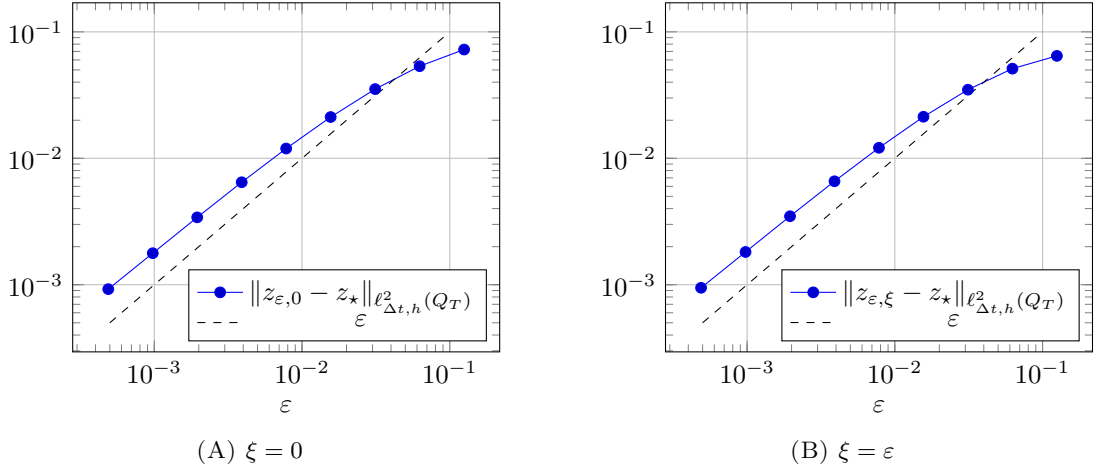

\begin{figure}
    \centering
    \begin{subfigure}{0.48\textwidth}
        \centering

        \begin{tikzpicture}
        \begin{axis}[,
            width=\linewidth,
            xmode=log,
            ymode=log,
            xlabel=$\varepsilon$,
            grid=major,
            legend pos=south east
        ]
        
        \addplot table [col sep=comma,x=eps,y=L2_error] {L2_error_xi_0_2d.csv};
        \addlegendentry{$\|z_{\varepsilon,0} - z_\star\|_{\ell^2_{\Delta t, h}(Q_T)}$ }
        
        \addplot[dashed,domain=5e-4:1e-1] {10*x};
        \addlegendentry{$10\varepsilon$}
        
        \end{axis}
        \end{tikzpicture}
        
        \caption{$\xi=0$}
        \label{subfig:1d_xi_0_2d}
    \end{subfigure}%
    \hfill
    \begin{subfigure}{0.48\textwidth}
        \centering
        \begin{tikzpicture}
        \begin{axis}[
            width=\linewidth,
            xmode=log,
            ymode=log,
            xlabel=$\varepsilon$,
            grid=major,
            legend pos=south east
        ]
        
        \addplot table [col sep=comma,x=eps,y=L2_error] {L2_error_xi_eps_2d.csv};
        \addlegendentry{$\|z_{\varepsilon,\xi} - z_\star\|_{\ell^2_{\Delta t, h}(Q_T)}$ }
        
        \addplot[dashed,domain=5e-4:1e-1] {10*x};
        \addlegendentry{$10\varepsilon$}
        
        \end{axis}
        \end{tikzpicture}
        
        \caption{$\xi=\varepsilon$}
        \label{subfig:1d_xi_eps_2d}
    \end{subfigure}
    
    \caption{Convergence for $\varepsilon\to 0$ and $\xi=0$ (Figure \ref{subfig:1d_xi_0_2d}) or $\xi=\varepsilon$ (Figure \ref{subfig:1d_xi_eps_2d}). $\mu=0.4$, $q=2.5$, $T=0.18$, $\Omega = (0,1)^2$, $\Delta t = 10^{-4}$, $\Delta x = 10^{-2}$.}
    \label{fig:cvg_2d}
\end{figure}
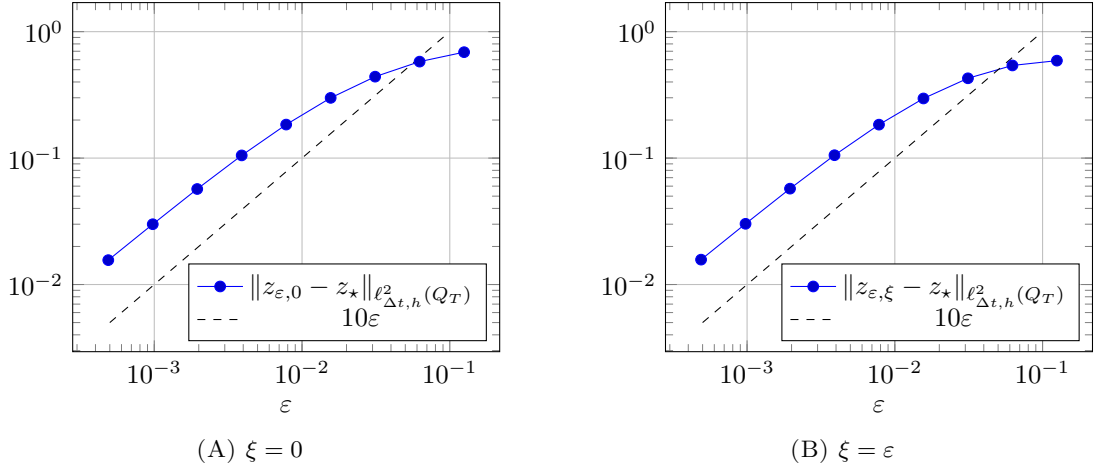

\subsection{Extinction phenomenon}
Solutions to the Cauchy–Dirichlet problem for the fast diffusion equation are known to extinguish in finite time; that is, for any initial data $z_0$, there exists a finite time $t_*(z_0)$ such that the solution of $(\mathrm{P})$ satisfies $z(x,t)=0$ for all $t > t_*(z_0)$ and $x \in \Omega$. In contrast, the numerical scheme \eqref{eq:scheme} does not appear to exhibit this extinction phenomenon. Indeed, from \eqref{eq:scheme} it follows that if $\mathbf{u}^{n+1} = \mathbf{v}^{n+1} = 0$, then $\mathbf{u}^{n} = \mathbf{v}^{n} = 0$. This implies that the only discrete solution extinguishing in finite time is the trivial one, corresponding to initial data $\mathbf{u}^0 = \mathbf{v}^0 = 0$. However, Figure \ref{fig:extinction} shows that $\|\mathbf{z}^n\|_{\ell^q_h(\Omega)}$ decays with an exponential rate past the extinction time $T_*$ of the true solution $z_*$. Here,
$$
\|\mathbf{z}^n\|_{\ell^q_h(\Omega)}^q= \sum_{i\in\mathcal{I}} |z_i^n|^q h^d.
$$

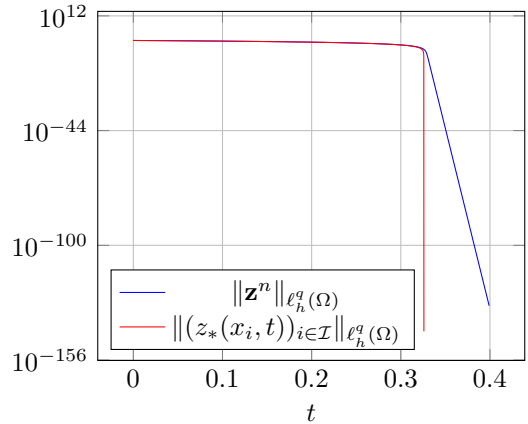
\begin{figure}
\centering

\begin{tikzpicture}
\begin{semilogyaxis}[
    width = 0.48\textwidth,
    xlabel={$t$},
    grid=major,
    legend pos=south west,
    legend image post style={mark=none}
]

\addplot table [col sep=comma, mark=none, x=t,y=Lq_norm] {Lq_norm.csv};
\addlegendentry{$\|\mathbf{z}^n\|_{\ell^q_h(\Omega)}$}

\addplot table [col sep=comma, mark=none, x=t,y=Lq_norm] {Lq_norm_true.csv};
\addlegendentry{$\|(z_*(x_i,t))_{i\in\mathcal{I}}\|_{\ell^q_h(\Omega)}$}

\end{semilogyaxis}
\end{tikzpicture}

\caption{Evolution of $\|\mathbf{z}^n\|_{\ell^q_h(\Omega)}$ and $\|(z_*(x_i,t))_{i\in\mathcal{I}}\|_{\ell^q_h(\Omega)}$. $\xi=\varepsilon=10^{-4}$, $\mu=0.5$, $q=2.5$, $\Omega=(0,1)$, $\Delta t=10^{-4}$, $h=10^{-2}$, and $u_0, v_0$ defined by \eqref{eq:ini_data}, \eqref{eq:data_ini_u0}, \eqref{eq:data_ini_v0}.}
\label{fig:extinction}
\end{figure}

The numerical results presented in this section demonstrate that the proposed reaction-diffusion system serves as a valid approximation framework for the fast diffusion equation, regardless of whether the system is parabolic-elliptic ($\xi=0$) or fully parabolic ($\xi>0$). Moreover, while the discrete scheme does not reproduce the finite-time extinction property in a strict sense, it successfully captures the rapid exponential decay of the solution beyond the extinction time of the limit problem. 

\section*{Acknowledgments}
The authors would like to thank Professor Goro Akagi for his valuable advice on the analysis of maximal monotone operators. 
HM was partially supported by JSPS KAKENHI Grant Numbers 24H00188 and 21KK0044.




\bibliography{references}

@book{barbu,
  author    = {Barbu, Viorel},
  title     = {Nonlinear Differential Equations of Monotone Types in Banach Spaces},
  publisher = {Springer},
  year      = {2010},
  series    = {Springer Monographs in Mathematics},
  address   = {New York},
  doi       = {10.1007/978-1-4419-5542-5}
}

@article {ninomiya2017rda,
    AUTHOR = {Ninomiya, Hirokazu and Tanaka, Yoshitaro and Yamamoto, Hiroko},
     TITLE = {Reaction, diffusion and non-local interaction},
   JOURNAL = {J. Math. Biol.},
  FJOURNAL = {Journal of Mathematical Biology},
    VOLUME = {75},
      YEAR = {2017},
    NUMBER = {5},
     PAGES = {1203--1233},
      ISSN = {0303-6812,1432-1416},
   MRCLASS = {35K57 (35A35 35B36 35K58 35R09 92C15 93A30)},
  MRNUMBER = {3694699},
       DOI = {10.1007/s00285-017-1113-x},
       URL = {https://doi.org/10.1007/s00285-017-1113-x},
}

@article{Murakawa2007,
doi = {10.1088/0951-7715/20/10/003},
url = {https://doi.org/10.1088/0951-7715/20/10/003},
year = {2007},
month = {},
publisher = {},
volume = {20},
number = {10},
pages = {2319},
author = {Murakawa, H},
title = {Reaction-diffusion system approximation to degenerate parabolic systems},
journal = {Nonlinearity}
}

@misc{murakawa_tanaka2024,
      title={Keller-Segel type approximation for nonlocal Fokker-Planck equations in one-dimensional bounded domain}, 
      author={Hideki Murakawa and Yoshitaro Tanaka},
      year={2024},
      eprint={2402.11511},
      archivePrefix={arXiv},
      primaryClass={math.AP},
      journal = {arXiv},
      url={https://arxiv.org/abs/2402.11511}, 
}

@article{evans,
  author  = {Evans, L. C.},
  title   = {A convergence theorem for a chemical diffusion-reaction system},
  journal = {Houston J. Math.},
  volume  = {6},
  number  = {2},
  pages   = {259--267},
  year    = {1980}
}

@incollection{ehhp,
  author    = {Eymard, R. and Hilhorst, D. and van der Hout, R. and Peletier, L. A.},
  title     = {A reaction-diffusion system approximation of a one-phase {S}tefan problem},
  booktitle = {Optimal Control and Partial Differential Equations},
  editor    = {Menaldi, J. L. and Rofman, E. and Sulem, A.},
  publisher = {IOS Press},
  address   = {Amsterdam},
  pages     = {156--170},
  year      = {2000}
}

@article{himn,
  author  = {Hilhorst, D. and Iida, M. and Mimura, M. and Ninomiya, H.},
  title   = {A competition-diffusion system approximation to the classical two-phase {S}tefan problem},
  journal = {Japan J. Indust. Appl. Math.},
  volume  = {18},
  number  = {2},
  pages   = {161--180},
  year    = {2001},
  doi     = {10.1007/BF03167406}
}

@incollection{in,
  author    = {Iida, M. and Ninomiya, H.},
  title     = {A reaction-diffusion approximation to a cross-diffusion system},
  booktitle = {Recent Advances on Elliptic and Parabolic Issues},
  editor    = {Chipot, M. and Ninomiya, H.},
  publisher = {World Scientific},
  pages     = {145--164},
  year      = {2006},
  doi       = {10.1142/9789812773951_0007}
}

@article{m_lscds,
	author = {Murakawa, H.},
	title = {A linear scheme to approximate nonlinear cross-diffusion systems},
	DOI= "10.1051/2an/2011010",
	url= "https://doi.org/10.1051/m2an/2011010",
	journal = {ESAIM: M2AN},
	year = 2011,
	volume = 45,
	number = 6,
	pages = "1141-1161",
}

@article{m_elscds,
  author  = {Murakawa, H.},
  title   = {An efficient linear scheme to approximate nonlinear diffusion problems},
  journal = {Japan J. Indust. Appl. Math.},
  volume  = {35},
  number  = {1},
  pages   = {71--101},
  year    = {2018},
  doi     = {10.1007/s13160-017-0279-3}
}

@article{skt,
  author  = {Shigesada, N. and Kawasaki, K. and Teramoto, E.},
  title   = {Spatial segregation of interacting species},
  journal = {J. Theoret. Biol.},
  volume  = {79},
  number  = {1},
  pages   = {83--99},
  year    = {1979},
  doi     = {10.1016/0022-5193(79)90258-3}
}

@article{IshiiMurakawaTanaka2026,
      title={Relationship between haptotaxis and chemotaxis in cell dynamics}, 
      author={Hiroshi Ishii and Hideki Murakawa and Yoshitaro Tanaka},
      year={2025},
journal = {arXiv, arxiv.org/abs/2503.00280},
      eprint={2503.00280},
      archivePrefix={arXiv},
      primaryClass={math.AP},
      url={https://arxiv.org/abs/2503.00280}, 
}

@book{VazquezPME2006,
  author    = {Juan Luis V{\'a}zquez},
  title     = {The Porous Medium Equation: Mathematical Theory},
  series    = {Oxford Mathematical Monographs},
  publisher = {Oxford University Press},
  address   = {Oxford},
  year      = {2006},
  isbn      = {9780198569039},
  doi       = {10.1093/acprof:oso/9780198569039.001.0001}
}

@book{VazquezSmoothing2006,
  author    = {Juan Luis V{\'a}zquez},
  title     = {Smoothing and Decay Estimates for Nonlinear Diffusion Equations: Equations of Porous Medium Type},
  publisher = {Oxford University Press},
  address   = {Oxford},
  year      = {2006},
  isbn      = {9780199202973},
  doi       = {10.1093/acprof:oso/9780199202973.001.0001}
}

@book{Rubinstein1971Stefan,
  author    = {L. I. Rubinshtein},
  title     = {The Stefan Problem},
  series    = {Translations of Mathematical Monographs},
  volume    = {27},
  publisher = {American Mathematical Society},
  address   = {Providence, RI},
  year      = {1971},
  isbn      = {978-1-4704-2850-1}
}

@article{AltLuckhaus1983,
  author  = {H. W. Alt and S. Luckhaus},
  title   = {Quasilinear elliptic-parabolic differential equations},
  journal = {Mathematische Zeitschrift},
  volume  = {183},
  number  = {3},
  pages   = {311--341},
  year    = {1983},
  doi     = {10.1007/BF01176474}
}

@article {DiBenedettoFriedman1984,
    AUTHOR = {DiBenedetto, Emmanuele and Friedman, Avner},
     TITLE = {The ill-posed {H}ele-{S}haw model and the {S}tefan problem for
              supercooled water},
   JOURNAL = {Trans. Amer. Math. Soc.},
  FJOURNAL = {Transactions of the American Mathematical Society},
    VOLUME = {282},
      YEAR = {1984},
    NUMBER = {1},
     PAGES = {183--204},
      ISSN = {0002-9947,1088-6850},
   MRCLASS = {35R35 (35K05 80A20)},
  MRNUMBER = {728709},
MRREVIEWER = {Antonio\ Fasano},
       DOI = {10.2307/1999584},
       URL = {https://doi.org/10.2307/1999584},
}

@article{BonforteFigalli24,
 author = {Bonforte, Matteo and Figalli, Alessio},
 title = {The {Cauchy}-{Dirichlet} problem for the fast diffusion equation on bounded domains},
 fjournal = {Nonlinear Analysis. Theory, Methods \& Applications. Series A: Theory and Methods},
 journal = {Nonlinear Anal., Theory Methods Appl., Ser. A, Theory Methods},
 issn = {0362-546X},
 volume = {239},
 pages = {55},
 note = {Id/No 113394},
 year = {2024},
 language = {English},
 doi = {10.1016/j.na.2023.113394},
}

@article{DamascelliGrossiPacella99,
title = {Qualitative properties of positive solutions of semilinear elliptic equations in symmetric domains via the maximum principle},
journal = {Annales de l'Institut Henri Poincar\'e C, Analyse non lin\'eaire},
volume = {16},
number = {5},
pages = {631-652},
year = {1999},
issn = {0294-1449},
doi = {https://doi.org/10.1016/S0294-1449(99)80030-4},
url = {https://www.sciencedirect.com/science/article/pii/S0294144999800304},
author = {Lucio Damascelli and Massimo Grossi and Filomena Pacella},
}

@article{BerrymanHolland80,
 author = {Berryman, James G. and Holland, Charles J.},
 title = {Stability of the separable solution for fast diffusion},
 fjournal = {Archive for Rational Mechanics and Analysis},
 journal = {Arch. Ration. Mech. Anal.},
 issn = {0003-9527},
 volume = {74},
 pages = {379--388},
 year = {1980},
 language = {English},
 doi = {10.1007/BF00249681},
 zbMATH = {3717792},
 Zbl = {0458.35046}
}

@book{Struwe,
 author = {Struwe, Michael},
 title = {Variational methods. {Applications} to nonlinear partial differential equations and {Hamiltonian} systems.},
 edition = {4th ed.},
 fseries = {Ergebnisse der Mathematik und ihrer Grenzgebiete. 3. Folge},
 series = {Ergeb. Math. Grenzgeb., 3. Folge},
 issn = {0071-1136},
 volume = {34},
 isbn = {978-3-540-74012-4},
 year = {2008},
 publisher = {Berlin: Springer},
 language = {English},
 doi = {10.1007/978-3-540-74013-1},
 zbMATH = {5203477},
 Zbl = {1284.49004}
}

@book{CH,
 author = {Cazenave, Thierry and Haraux, Alain},
 title = {An introduction to semilinear evolution equations. {Transl}. by {Yvan} {Martel}.},
 edition = {Revised ed.},
 fseries = {Oxford Lecture Series in Mathematics and its Applications},
 series = {Oxf. Lect. Ser. Math. Appl.},
 volume = {13},
 isbn = {0-19-850277-X},
 year = {1998},
 publisher = {Oxford: Clarendon Press},
 language = {English},
 zbMATH = {1219622},
 Zbl = {0926.35049}
}
\bibliographystyle{abbrv}

\end{document}